\documentclass[a4paper,10pt]{amsart}
 \usepackage{amsfonts,mathrsfs,amsrefs}
 \usepackage{amsthm}
 \usepackage{amssymb}
 \usepackage{enumerate}
 \usepackage{tikz-cd}
 \usepackage{tikz}
 \usepackage{cleveref}
 \usepackage{url}
 \usepackage{framed}
\usepackage{enumitem}

\theoremstyle{definition}
\newtheorem{Def}{Definition}[section]
\newtheorem{ex}[Def]{Example}

\newtheorem{rem}[Def]{Remark}

\theoremstyle{plain}
\newtheorem{prop}[Def]{Proposition}

\newtheorem{thm}[Def]{Theorem}
\newtheorem*{thm*}{Theorem}
\newtheorem{claim*}{Claim}
\newtheorem{lem}[Def]{Lemma}
\newtheorem{cor}[Def]{Corollary}
\newtheorem*{cor*}{Corollary}

\newtheorem*{con*}{Conjecture}
\newtheorem{frag}[Def]{Question}
\newtheorem*{frag*}{Question}

\newtheorem*{verm*}{Vermutung}

\usepackage{mathrsfs}

\newcommand{\rank}{\operatorname{rank}} 

\newcommand{\Div}{\operatorname{Div}}

\newcommand{\cF}{{\mathcal F}}

\newcommand{\cH}{{\mathcal H}}
\newcommand{\cI}{{\mathcal I}}

\newcommand{\cO}{{\mathcal O}}

\newcommand{\cR}{{\mathcal R}}

\newcommand{\cV}{{\mathcal V}}

\newcommand{\C}{{\mathbb C}}
\newcommand{\R}{{\mathbb R}}
\newcommand{\pp}{\mathbb{P}}

\newcommand{\N}{{\mathbb N}}
\newcommand{\Z}{{\mathbb Z}}

\title[Determinantal Representations of Adjoints]{Adjoints of Polytopes: Determinantal Representations and Smoothness}

\author{Clemens Brüser}
\address{Technische Universit\"at Dresden, Germany} 
\email{clemens.brueser@tu-dresden.de}
\author{Mario Kummer}
\email{mario.kummer@tu-dresden.de}
\author{Dmitrii Pavlov}
\email{dmitrii.pavlov@mis.mpg.de}

\thanks{The first and second authors were supported by the DFG grant 502861109.}
\newcommand{\comment}[1]{}

\setlist[enumerate]{label=(\roman*)}

\begin{document}


\begin{abstract}
 In this article we study determinantal representations of adjoint hypersurfaces of polytopes. We prove that adjoint polynomials of all polygons can be represented as determinants of tridiagonal symmetric matrices of linear forms with the matrix size being equal to the degree of the adjoint. We prove a sufficient combinatorial condition for a surface in the projective three-space to have a determinantal representation and use it to show that adjoints of all three-dimensional polytopes with at most eight facets and a simple facet hyperplane arrangement admit a determinantal representation. This includes all such polytopes with a smooth adjoint. We demonstrate that, starting from four dimensions, adjoint hypersurfaces may not admit linear determinantal representations. Along the way we prove that, starting from three dimensions, adjoint hypersurfaces are typically singular, in contrast to the two-dimensional case. We also consider a special case of interest to physics, the ABHY associahedron. We construct a determinantal representation of its universal adjoint in three dimensions and show that in higher dimensions a similarly structured representation does not exist. 
\end{abstract}
\maketitle
\section{Introduction}
\emph{Positive geometry} \cite{ranestad2025positive} is a rapidly emerging area on the intersection of mathematics and theoretical physics. One of its main goals is to use tools from applied algebraic geometry and combinatorics to answer questions appearing in particle physics and cosmology. Central mathematical objects in this context are \emph{positive geometries}, introduced in \cite{arkani2017positive}. The formal definition of a positive geometry (see e.g. \cite{Lam2022PosGeom}*{Definition 1}) is involved. Roughly speaking, this is a pair consisting of a projective complex variety $X$ and a semialgebraic set $X_{\geq 0}$ inside the real points of $X$ that comes with a unique rational top-form $\Omega_{X,X_{\geq 0}}$ on $X$, called the \emph{canonical form}, which encapsulates information about the boundary structure of $X_{\geq 0}$. 

Arguably one of the simplest and most well-studied examples of positive geometries is that of polytopes in the projective space $\mathbb{P}^n$. For a polytope $P\subseteq \mathbb{P}^n$ the canonical form is given by 
$$\Omega_{\mathbb{P}^n, P} = f_P(x) \mu_{\mathbb{P}^n}(x).$$
Here $\mu_{\mathbb{P}^n} = \sum_{i=0}^n (-1)^{i}x_i dx_1\wedge\ldots\wedge \widehat{dx_i} \wedge\ldots\wedge dx_n$ is the standard meromorphic top-form on $\mathbb{P}^n$, where $(x_0:\ldots:x_n)$ is the vector of homogeneous coordinates, and $f_P(x)$ is a rational function which we will refer to as the \emph{canonical function} of $P$. 

The structure of $f_P(x)$ can be read off from the combinatorics of the polytope $P$. To explain this, we recall the notion of the \emph{residual arrangement}. This is defined as the union of all intersections of facet hyperplanes of $P$ that do not contain a face of $P$. The canonical function is then of the form 
$$f_P(x) = \dfrac{\alpha_P(x)}{\prod_{i=1}^k l_i(x)},$$
where $l_i(x)$ are the linear forms defining the facet hyperplanes of $P$, and $\alpha_P(x)$ is the \emph{adjoint polynomial} of $P$, sometimes simply called the \emph{adjoint}. This is a polynomial of degree $k-n-1$ that vanishes on the residual arrangement of $P$. If $P$ is a polytope whose facet hyperplane arrangement is simple, then up to a scalar factor the adjoint is the unique polynomial with this property by \cite{KohnRanestad2019AdjCurves}*{Theorem 1}. The hypersurface defined by the adjoint polynomial is called the \emph{adjoint hypersurface}. 

The primary application of positive geometries is in particle physics, where they are used to compute scattering amplitudes. These fundamental quantities encode the probability that a certain particle scattering process occurs. Already polytopes are of relevance for physics. A central example is the \emph{ABHY associahedron}\footnote{The term ABHY refers to the initials of the authors Arkani-Hamed, Bai, He, and Yan of [2].}, introduced in \cite{arkani2018scattering}. The canonical form of this polytope encodes the tree-level planar bi-adjoint scalar $n$-point amplitude in $\phi^3$ quantum field theory.

While using positive geometries conceptually simplifies the computation of scattering amplitudes, computing the canonical form can still be quite involved, and this approach potentially obscures some of the non-obvious structure of the amplitude. It is clear that the complexity is concentrated in the numerator of the canonical function, namely in the adjoint polynomial: the remaining parts are known from a description of the polytope by inequalities. It is therefore a, both mathematically and physically, interesting question to investigate the structure of adjoints. 
In the physics reference \cite{arkani2024hidden}*{Section 8} the authors raise the question of whether adjoints have determinantal representations, that is, can be written as determinants of matrices of linear forms. In this article we offer answers to this question. 

The question of whether a homogeneous polynomial can be written as the determinant of a square matrix with linear entries has been extensively studied in algebraic geometry and is closely related to the geometry of the hypersurface $X$ in $\pp^n$ that the polynomial defines, see e.g. \cite{Beauville2000DetHypSurf} for a modern treatment and overview. For example, if $n\geq4$ and $X$ is smooth, then such a determinantal representation does not exist \cite{Dolgachev2012CAG}*{Section 4.1.1}. By constructing a four-dimensional polytope that has a smooth adjoint hypersurface, we show that adjoints do in general not have a determinantal representation. In the plane, i.e., $n=2$, it is known that a determinantal representation always exists: this follows e.g. from \cite{EisenbudSchreyerWeyman2003UlrichBundles}*{Corollary 4.5}. However, we prove that adjoints of polygons admit a particularly nice one: it is real, symmetric, it is tridiagonal in the sense that its only non-zero entries are on the main diagonal or on the first super- or subdiagonal, and it is positive definite on the interior of the polygon. A simple count of parameters shows that a tridiagonal determinantal representation is rare even for planar curves and the definitness imposes serious restrictions on the topology of the real part of the curve \cite{heltonvinnikov}. Finally, smooth hypersurfaces in $\pp^3$ can have a determinantal representation but starting from degree four this happens only for a subset of measure zero: the existence of a determinantal representation for a smooth hypersurface $X$ of degree $d$ in $\pp^3$ is equivalent to the existence of a smooth curve $C\subseteq X$ of a certain degree and genus (depending on $d$) \cite{Beauville2000DetHypSurf}*{Proposition 6.2} and the Noether--Lefschetz theorem \cite{noetherlefschetz} implies that a very general surface of degree $d$ does not contain such a curve. We prove a sufficient criterion for the existence of a determinantal representation of a similar flavor that is tailored towards adjoints of polytopes: instead of a certain smooth curve, we require the existence of a certain line arrangement on $X$. We use it to show that adjoints of all three-dimensional polytopes with at most eight facets and a simple facet hyperplane arrangement admit a determinantal representation. This includes all such polytopes with a smooth adjoint.

We now give a detailed overview of the structure and main results of the article. In \Cref{sec:2} we give preliminaries on adjoint polynomials and hypersurfaces. In \Cref{sec:detprelim} we recall the basics of determinantal representations.

Since the existence of determinantal representations of a polynomial is connected to the (non-)smoothness of the hypersurface defined by it, in \Cref{sec:smooth} we investigate when the adjoint hypersurface of a polytope is smooth. We give a complete classification of three-dimensional polytopes with simple facet hyperplane arrangement that can have smooth adjoints in \Cref{thm:smooth3d}. In \Cref{thm:smoothness} we prove that in each dimension starting from three there are only finitely many combinatorial types of polytopes for which the adjoint can be smooth. In \Cref{ex:counter} we present  a four-dimensional polytope with a smooth adjoint hypersurface, showing that in dimension four and higher not every adjoint has a determinantal representation.

The counterexample from \Cref{sec:smooth} motivates us to concentrate our attention on low dimensions. In \Cref{sec:3} we study determinantal representations of adjoints of polygons in the plane. Our main result here is \Cref{thm:recursiveDetRep} stating that the adjoint polynomial of every polygon $P$ has a determinantal representation by a symmetric tridiagonal matrix of linear forms that possesses additional recursive structure: its subdeterminants are adjoints of subpolygons of $P$. We note that this result can be used to construct the entries of the determinantal representation of $\alpha_P(x)$ (up to scaling of each entry) without knowing the polynomial beforehand, see \Cref{rem:detRepExpl}.

In \Cref{sec:4} we turn to the three-dimensional case. Our key instrument here is what we call \emph{nice arrangements of lines} (\Cref{def:nicearrangements}). In \Cref{thm:nicearrangementsdetrep} we prove a general result showing that if an algebraic surface in $\mathbb{P}^3$ of degree $D$ contains a nice line arrangement for degree $D$, then the defining polynomial of this surface has a determinantal representation that can be constructed explicitly. We use this result to show that the adjoints of all three-dimensional polytopes with at most eight facets and a simple facet hyperplane arrangement have a determinantal representation in \Cref{cor:smoothdetrep}. We note that the result can also be applied in the case of arbitrary polytopes, without the simplicity constraint on the facet hyperplane arrangement.  

Finally, \Cref{sec:6} is devoted to our motivating example from physics, the ABHY associahedron. We construct simple determinantal representations of universal (as defined in \cite{telen2025toric}) adjoints of two- and three-dimensional ABHY associahedra in \Cref{ex:3dassoc}. In \Cref{thm:noavrep} we prove a negative result showing that in dimension at least four a similarly structured determinantal representation does not exist. Therefore, while we do not rule out that in this higher-dimensional case some determinantal representation of the universal adjoint exists, we are skeptical that such a representation can have an illuminating structure. 

\section{Preliminaries on Adjoints} \label{sec:2}
This article studies determinantal representations of a special class of hypersurfaces. In this section, we introduce these hypersurfaces.
We work in the complex projective space $\mathbb{P}^n$. A \emph{projective polytope} $P\subseteq \mathbb{P}^n$ is a convex full-dimensional polytope inside the real points of some affine chart of $\mathbb{P}^n$. In what follows, for brevity we omit the word ``projective''. Let $\mathcal{H}_P$ denote the facet hyperplane arrangement, i.e., the hyperplane arrangement in $\mathbb{P}^n$ obtained by taking the Zariski closures of all the facets of $P$. We denote the set of vertices of $P$ by $V(P)$. We will also need the following definition. 

\begin{Def}[Residual arrangements]
    The \emph{residual arrangement} $\cR(P)$ of a polytope $P\subseteq \mathbb{P}^n$ is the collection of all intersections of hyperplanes in $\mathcal{H}_P$ that do not contain a face of $P$.
\end{Def}

Let $P$ be a polytope in $\mathbb{P}^n$. Fix an affine chart containing $P$ with coordinates $t_1,\ldots,t_n$. In \cite{warren1996barycentric} Warren defined the \emph{adjoint polynomial} $\mathrm{adj}_P$ of $P$ by fixing a triangulation $\tau(P)$ of $P$ and setting
$$\mathrm{adj}_P(t) = \sum\limits_{\sigma\in\tau(P)}\mathrm{vol}(\sigma)\prod\limits_{v\in V(P)\smallsetminus V(\sigma)}\ell_v(t),$$
where $\ell_v(t) = 1-\sum\limits_{i=1}^n v_it_i$. Warren also showed that this polynomial is independent of the choice of the triangulation $\tau(P)$. Note that this polynomial can be naturally homogenized by replacing the constant term $1$ with $t_0$ in each linear polynomial $\ell_v(t)$. In what follows, $\mathrm{adj}_P$ is this homogenized polynomial.  

In \cite{KohnRanestad2019AdjCurves}*{Proposition 2} Kohn and Ranestad showed that for any polytope $P$ the zero locus of the adjoint polynomial $\mathrm{adj}_{P^\circ}$ of the dual polytope $P^\circ$ contains the residual arrangement of $P$. Moreover, they showed that if $\mathcal{H}_P$ is a \emph{simple} arrangement of $k$ hyperplanes in $\mathbb{P}^n$ (that is, if the intersection of any $i$ hyperplanes in $\mathcal{H}_P$ has codimension $i$ for $i\leq n$ and is empty for $i>n$), the zero locus of $\mathrm{adj}_{P^\circ}$ is the \emph{unique} hypersurface of degree $k-n-1$ containing the residual arrangement of $P$ \cite{KohnRanestad2019AdjCurves}*{Theorem 1}.

\begin{rem}
    Note that if $\mathcal{H}_P$ is a simple arrangement, then $P$ is a simple polytope (that is, each of its vertices is contained in exactly dimension many facets) but the converse does not hold; for instance, the standard three-dimensional cube is a counterexample because four facet hyperplanes meet at a point at infinity. Another counterexample is given by the standard realization of the three-dimensional permutohedron where such bad points even exist in $\R^3$.
\end{rem}

From now on we will mostly consider polytopes $P$ such that $\mathcal{H}_P$ is a simple arrangement.

\begin{Def}[Adjoint hypersurfaces and polynomials]
    To avoid switching back and forth between $P$ and $P^\circ$, we define the \emph{adjoint hypersurface} $A_P$ of a polytope $P$ with $k$ facets and a simple facet hyperplane arrangement $\mathcal{H}_P$ to be the unique hypersurface of degree $k-n-1$ vanishing on the residual arrangement of $P$ (and not of $P^\circ$) and the \emph{adjoint polynomial} $\alpha_P$ to be the defining polynomial of $A_P$.
\end{Def}

Note that $\alpha_P$ is only defined up to a scalar factor. In terms of Warren's adjoint we then have (up to this scalar factor) $\alpha_P := \mathrm{adj}_{P^\circ}$. A formula for $\alpha_P$ in terms of the facet equations of $P$ is given in \cite{PavlovTelen2024PosGeom}*{Proposition 2.1}. This choice of definition is also natural from the point of positive geometries since the polynomial $\alpha_P$ is the numerator of the canonical from of $P$ \cite{gaetz2020positive}*{Theorem 3.2} (cf. \cite{Lam2022PosGeom}*{Theorem 5}). In this spirit, we define the adjoint polynomial $\alpha_P$ of an arbitrary convex polytope $P$ to be the numerator of its canonical form. Note that any polytope $P$ can be approximated by a sequence of polytopes with simple facet hyperplane arrangements. Then the adjoint of $P$ is the limit of the adjoints of the polytopes in the approximating sequence. This follows from \cite{KohnRanestad2019AdjCurves}*{Corollary 1}. 

\section{Preliminaries on Determinantal Representations} \label{sec:detprelim}
We now recall the basics of determinantal representations. 
\begin{Def}[Determinantal representations]\label{def:detrep}
    Let $f\in\C[x_1,\ldots,x_n]$ be a polynomial of degree $d$ and let $K\subseteq\C$ be a subfield. 
    \begin{enumerate}
        \item We say that a $d\times d$ matrix $M(x)$, whose entries are polynomials of degree at most one with coefficients in $K$, is a \emph{determinantal representation of $f$ over $K$} if $f=\det M(x)$. If we do not specify the field $K$, then we mean a determinantal representation over $\C$.
        \item A determinantal representation $M(x)$ is called \emph{symmetric} if $M(x)$ is a symmetric matrix.
        \item A determinantal representation $M(x)$ is called \emph{homogeneous} if $M(0)=0$.
        \item A determinantal representation $M(x)$ over $\R$ is called \emph{definite} (at $a\in\R^n$) if it is symmetric and $M(a)$ is positive or negative definite.
    \end{enumerate}
\end{Def}
\begin{rem}\label{rem:detrep}
Let $f\in\C[x_1,\ldots,x_n]$ be a polynomial of degree $d$.
\begin{enumerate}
    \item Let $M(x)=M_0+M_1(x)$ be a determinantal representation of $f$ where $M_0$ is a scalar matrix and $M_1(x)$ has homogeneous linear forms as entries. The matrix $x_0\cdot M_0+M(x)$ is a  determinantal representation of the \emph{homogenization} $f^{\rm h}=x_0^d\cdot f(\frac{x_1}{x_0},\cdots,\frac{x_n}{x_0})$ of $f$.  If $f$ is homogeneous, then $M_1(x)$ is a determinantal representation of $f$. Therefore, if a homogeneous polynomial has a determinantal representation, then it also has a homogeneous one.
    \item Conversely, if $M(x_0,\ldots,x_n)$ is a determinantal representation of $f^{\rm h}$, then $M(1,x_1,\ldots,x_n)$ is a determinantal representation of $f$.
    \item There always exists a matrix $N(x)$ of some size $e\gg d$ with polynomials of degree at most one as entries such that $f=\det N(x)$ \cite{valiant}. Note that the homogenization of $N(x)$ as above is not a determinantal representation of the homogenization $f^{\rm h}$ itself but rather of $x_0^{e-d}\cdot f^{\rm h}$. 
\end{enumerate}    
\end{rem}
By the previous remark it suffices to consider homogeneous polynomials.
The existence of a determinantal representation of a homogeneous polynomial $f\in\C[x_0,\ldots,x_n]$ is closely related to the geometry of its zero set $\cV(f)\subseteq\pp^n$ in complex projective space $\pp^n=\pp^n(\C)$. Before we describe the general theory, we start with the well-known and instructive case of quadratic polynomials for which we provide an elementary proof.

\begin{prop}\label{prop:quadratic}
    Let $X\subseteq\pp^n$ be a quadratic hypersurface cut out by the homogeneous quadratic polynomial $f\in\C[x_0,\ldots,x_n]$. The following are equivalent:
    \begin{enumerate}
        \item $f$ has a determinantal representation;
        \item $X$ contains a linear subspace of dimension $n-2$. 
    \end{enumerate}
    Furthermore, if $f$ has a determinantal representation and $n\geq4$, then $X$ is singular.
\end{prop}
\begin{proof}
    For the direction ``$(i)\Rightarrow(ii)$'' let
    \begin{equation*}
        M(x)=\left(\begin{matrix}
            l_{11}(x)& l_{12}(x)\\ l_{21}(x)& l_{22}(x)
        \end{matrix}\right)
    \end{equation*}
    where the $l_{ij}(x)$ are linear forms such that
    \begin{equation}\label{eq:quadraticexpansion}
        f=\left(l_{11}(x)\cdot l_{22}(x) - l_{12}(x)\cdot l_{21}(x)\right).
    \end{equation}
    Then $X$ contains the linear subspace $\cV(l_{11},l_{12})$ which has dimension at least $n-2$. For the direction ``$(ii)\Rightarrow(i)$'' we recall that the maximal dimension of an isotropic subspace of $\C^{n+1}$, i.e., a subspace on which $f$ vanishes entirely, is $\dim(V_0)+\left\lfloor\frac{\rank(f)}{2}\right\rfloor$ where $V_0$ is the degenerate space of $f$. Our assumption implies that
    \begin{equation*}
        n-1\leq\dim(V_0)+\left\lfloor\frac{\rank(f)}{2}\right\rfloor=\dim(V_0)+\rank(f)-\left\lceil\frac{\rank(f)}{2}\right\rceil=n+1-\left\lceil\frac{\rank(f)}{2}\right\rceil.
    \end{equation*}
    Here the first inequality holds because if $X$ contains a linear subspace of $\pp^n$ of dimension $n-2$, then the corresponding subspace of $\C^{n+1}$ of dimension $n-1$ is isotropic with respect to $f$. The resulting inequality
    \begin{equation*}
        n-1\leq n+1-\left\lceil\frac{\rank(f)}{2}\right\rceil
    \end{equation*}
    shows that $\rank(f)\leq 4$ and thus, after a linear change of coordinates, that
    \begin{equation*}
        f=x_0^2+x_1^2-x_2^2-x_3^2
    \end{equation*}
    and we can write
    \begin{equation*}
        f=\det\left(\begin{matrix}
            x_0+x_2& x_1+x_3\\ -x_1+x_3& x_0-x_2
        \end{matrix}\right)
    \end{equation*}
    which shows $(i)$. We further note that if $n\geq4$, then every point $(p_0:\cdots:p_n)\in\pp^n$ with $p_0=\cdots=p_3=0$ is in the singular locus of $X$.  
\end{proof}

\subsection{Plane hyperbolic curves}
It was shown by Helton and Vinnikov \cite{heltonvinnikov} that a homogeneous polynomial $0\neq f\in\R[x_0,x_1,x_2]$ has (up to sign) a definite determinantal representation if and only if it is hyperbolic in the following sense.
\begin{Def}[Hyperbolic polynomials]
 A homogeneous real polynomial $f\in\R[x_1,\ldots,x_n]$ is called \emph{hyperbolic} with respect to $e\in\R^n$ if for every $a\in\R^n$ the univariate polynomial $f(te-a)$ has only zeros in $\R$.
\end{Def}
In the following we describe a criterion due to Plaumann and Vinzant \cite{PlaumannVinzantHyp2013LDR} for a symmetric determinantal representation to be definite. For this we need some preparation.
Recall that the complement of an embedded circle $S$ in the real projective plane $\pp^2(\R)$ consists of either one or two connected components. In the former case, we call $S$ a \emph{pseudoline} and in the latter case an \emph{oval}. The complement of an oval has exactly one connected component that is homeomorphic to an open disc. We call it the \emph{interior} of the oval $S$. The other connected component is called the \emph{exterior} of $S$. The topology of the real projective zero set of ternary hyperbolic polynomials can be described as follows.
\begin{prop}[\cite{heltonvinnikov}*{Section 5}] \label{prop:hyperbRealLocus}
 Let $f\in\R[x_0,x_1,x_2]$ be hyperbolic with respect to $e\in\R^3$ and assume that no singularity of $X=\cV(f)\subseteq\pp^2$ lies in the real locus $X(\R)$.
 \begin{enumerate}
  \item If $\deg(f)=2m$, then $X(\R)$ has $m$ connected components $O_1,\ldots,O_m$. Each of them is an oval and $O_i$ is contained in the interior of $O_{i+1}$ for $i=1,\ldots,m-1$.
  \item If $\deg(f)=2m+1$, then $X(\R)$ has $m+1$ connected components, namely $m$ ovals $O_1,\ldots,O_m$ and one pseudoline $P$. Each $O_i$ is contained in the interior of $O_{i+1}$ for $i=1,\ldots,m-1$. The pseudoline $P$ is contained in the exterior of $O_m$.
 \end{enumerate}
 Moreover, if $\deg(f)\geq2$, the point $[e]\in\pp^2(\R)$ is contained in the interior of $O_1$ which is called the \emph{innermost oval}.
 In both cases,  the polynomial $f$ is called \emph{strictly hyperbolic} with respect to $e$.
\end{prop}

\begin{Def}[Interlacers]
    Let $f,g\in\R[x_0,x_1,x_2]$ be strictly hyperbolic with respect to $e$ and $d=\deg(f)=\deg(g)+1\geq2$. Let $X=\cV(f)$ and $Y=\cV(g)$. We say that $g$ is an \emph{interlacer} of $f$, or $g$ \emph{interlaces} $f$, if the closure of every connected component of $\pp^2(\R)\smallsetminus X(\R)$, except for the interior of the innermost oval of $X(\R)$, contains exactly one connected component of $Y(\R)$.
\end{Def}

Now we can formulate the criterion for definiteness.
\begin{thm}[\cite{PlaumannVinzantHyp2013LDR}*{Theorem 3.3}]\label{thm:interlmakesdef}
    Let $f,g\in\R[x_0,x_1,x_2]$ be strictly hyperbolic with respect to $e$ and $d=\deg(f)=\deg(g)+1\geq2$. Assume further that $X=\cV(f)\subseteq\pp^2$ is smooth and that $g$ interlaces $f$. Let $M$ be a symmetric determinantal representation of $f$ over $\R$ such that $g$ is the determinant of the top-left $(d-1)\times(d-1)$ submatrix of $M$. Then $M$ is a definite determinantal representation of $f$.
\end{thm}

\subsection{Divisors}
Some proofs will need the notion of \emph{divisors} from algebraic geometry. We refer to \cite{Hartshorne1977AG}*{Section II.6} as a general reference.
\begin{Def}[Divisors]\label{def:divisors}
    Let $X$ be an irreducible projective variety over $\C$ whose singular locus has codimension at least two. A \emph{prime divisor} on $X$ is a closed irreducible subvariety of codimension one. A \emph{(Weil) divisor} is an element of the free abelian group $\Div(X)$ generated by the prime divisors, i.e., a divisor is a formal sum $D=\sum n_i Y_i$ where the $Y_i$ are prime divisors and the $n_i$ are integers, only finitely many of which are non-zero. If all $n_i\geq0$, then $D$ is called \emph{effective}. This is denoted by $D\geq0$. Finally, the \emph{support} of $D$ is the set of prime divisors $Y_i$ for which $n_i\neq0$.
\end{Def}
\begin{ex}\label{ex:hyperplanedivisor}
Let $X$ be as in \Cref{def:divisors}.
\begin{enumerate}
    \item If $X$ is a curve, then the prime divisors on $X$ are exactly the points on $X$.
    \item If $X\subseteq\pp^n$, then every hypersurface $Y\subseteq\pp^n$ defines a divisor $Y.X$ on $X$ given by the sum of the irreducible components of $X\cap Y$ weighted by their multiplicities, and similarly a homogeneous polynomial $f$ defines the divisor $f.X$ on $X$.
\end{enumerate}
\end{ex}
We will use the following definition frequently in \Cref{sec:3}.
\begin{Def}[Contact curve]
    Let $X\subseteq\pp^2$ be a smooth curve. A curve $Y\subseteq\pp^2$ is called a \emph{contact curve} to $X$ with \emph{contact divisor} $D$ if the divisor $Y.X$ defined in \Cref{ex:hyperplanedivisor} is equal to $2D$ where $D$ is a divisor on $X$.
\end{Def}

\subsection{Determinantal representations and coherent sheaves}
The existence of a determinantal representation of a homogeneous polynomial $f$ is closely related to the existence of certain divisors on the hypersurface $\mathcal{V}(f)$. For a hypersurface $X$ in $\pp^3$ the precise condition is formulated in the following theorem in terms of the cohomology of the ideal sheaf of a curve on $X$. As a general reference for cohomology of coherent sheaves, see \cite{Hartshorne1977AG}*{Chapter III}. However, we actually recommend the reader less familiar with algebraic geometry to take \Cref{thm:idealsheafdetrep} and also \Cref{thm:nicearrangementsdetrep}, where we show that certain line arrangements satisfy the assumptions of \Cref{thm:idealsheafdetrep}, as a black box. Given a suitable curve, the determinantal representation can be easily computed from its vanishing ideal using the software \texttt{Macaulay2} \cite{M2}, see \Cref{rem:computedetrep}.
\begin{thm}\label{thm:idealsheafdetrep}
    Let $X\subseteq\pp^3$ be a hypersurface of degree $D$ with defining polynomial~$f$. Assume that there exists a closed subscheme  $C\subseteq X$ of dimension one whose ideal sheaf $\cI_C$ on $\pp^3$ satisfies the following conditions:
    \begin{align*}
        h^0(\pp^3,\cI_C(D-2))&=h^1(\pp^3,\cI_C(D-2))\\
        &=h^1(\pp^3,\cI_C(D-3))=0,\\
        h^2(\pp^3,\cI_C(D-3))&=h^3(\pp^3,\cI_C(D-3)).
    \end{align*}
    Then the free resolution of the sheaf $\cF=(\cI_C|_X)(D-1)$ has the form
    \begin{equation*}
    \begin{tikzcd}
        0\arrow{r}&\cO_{\pp^3}(-1)^D\arrow[r,"M"]&\cO_{\pp^3}^D\arrow{r}&\cF\arrow{r}&0
    \end{tikzcd}
    \end{equation*}
    where $M$ is a determinantal representation of $f$.    
    The ideal sheaf $\cI_C$ furthermore satisfies the following conditions:
     \begin{align*}
        h^0(\pp^3,\cI_C(k+D-1))&=0&\textrm{ for }k<0,\\
        h^1(\pp^3,\cI_C(k+D-1))&=0&\textrm{ for }k\in\Z,\\
        h^2(\pp^3,\cI_C(k+D-1))&=h^3(\pp^3,\cI_C(k+D-1))&\textrm{ for }k\geq-2,\\
        h^0(\pp^3,\cI_C(D-1))&=D.
    \end{align*}
\end{thm}
\begin{rem}\label{rem:computedetrep}
    Given the polynomial $f$ and defining equations for $C$, one can easily compute the free resolution of $\cF$, and thus a determinantal representation of~$f$, using the computer algebra system \texttt{Macaulay2} \cite{M2}. We will provide code and give examples at the end of \Cref{sec:detreps3d}.
\end{rem}
\begin{rem}
    If $X$ is smooth, then it follows from \cite{Beauville2000DetHypSurf}*{Proposition 6.2} that a projectively normal smooth curve $C$ of degree $\frac{1}{2}D(D-1)$ and genus $\frac{1}{6}(D-2)(D-3)(2D+1)$ satisfies the conditions of \Cref{thm:idealsheafdetrep}, and, moreover, the existence of such a curve on $X$ is also necessary for the existence of a determinantal representation of $f$.
\end{rem}
\begin{proof}
    Letting $\cF=(\cI_C|_X)(D-1)$ we obtain the short exact sequence
    \begin{equation*}
        0\to\cO_{\pp^3}(-1)\to\cI_C(D-1)\to\cF\to0.
    \end{equation*}
where the first map is multiplication by $f$. The long exact sequence in cohomology together with $h^1(\pp^3,\cO_{\pp^3}(j))=h^2(\pp^3,\cO_{\pp^3}(j))=0$ for all $j\in\Z$ gives:
\begin{align}
        h^0(\cF(k))&=h^0(\cI_C(k+D-1))-h^0(\cO_{\pp^3}(k-1)),\label{eq:coho0}\\
        h^1(\cF(k))&=h^1(\cI_C(k+D-1)),\label{eq:coho01}\\
        h^2(\cF(k))&=h^2(\cI_C(k+D-1))-h^3(\cI_C(k+D-1))+h^3(\cO_{\pp^3}(k-1))\label{eq:coho02}
    \end{align}
    for all $k\in\Z$. For $k=-1$ we obtain from \Cref{eq:coho0} and (\ref{eq:coho01}) that
    \begin{align*}
        h^0(\cF(-1))&=h^0(\cI_C(D-2))-h^0(\cO_{\pp^3}(-2))=0,\\
        h^1(\cF(-1))&=h^1(\cI_C(D-2))=0
    \end{align*}
    by our assumption and because $h^0(\cO_{\pp^3}(-2))=0$. For $k=-2$ we obtain from \Cref{eq:coho01} and (\ref{eq:coho02})
    \begin{align*}
        h^1(\cF(-2))&=h^1(\cI_C(D-3))=0,\\
        h^2(\cF(-2))&=h^2(\cI_C(D-3))-h^3(\cI_C(D-3))+h^3(\cO_{\pp^3}(-3))=0
    \end{align*}
    by our assumption and because $h^3(\cO_{\pp^3}(-3))=h^0(\cO_{\pp^3}(-1))=0$ by Serre duality. Thus the sheaf $\cF$ satisfies condition (b) of \cite{EisenbudSchreyerWeyman2003UlrichBundles}*{Proposition 2.1}. This implies that
    \begin{align*}
        h^0(\cF(k))=0&\textrm{ for }k<0,\\
        h^1(\cF(k))=0&\textrm{ for }k\in\Z,\\
        h^2(\cF(k))=0&\textrm{ for }k\geq-2,
    \end{align*}
    the statement about the free resolution and also about the determinantal representation of $f$ by parts (a) and (d) of \cite{EisenbudSchreyerWeyman2003UlrichBundles}*{Proposition 2.1}, respectively.
    Again using Serre duality and $h^0(\cO_{\pp^3}(j))=0$ for $j<0$, we obtain the first three additional claims on  $\cI_C$ from \Cref{eq:coho0}, (\ref{eq:coho01}) and (\ref{eq:coho02}). For the last one, we first observe that $h^0(\cF)=D$ follows from long exact sequence in cohomology obtained from the free resolution of $\cF$. The equality $h^0(\cI_C(D-1))=h^0(\cF)$ then follows from \Cref{eq:coho0} and $h^0(\cO_{\pp^3}(-1))=0$. Hence $h^0(\cI_C(D-1))=D$.
\end{proof}

\section{Smoothness of Adjoints} \label{sec:smooth} 
Understanding whether a hypersurface $X\subseteq\mathbb{P}^n$ has a determinantal representation is connected to understanding whether $X$ is smooth. More precisely, if $n\geq 4$ and $X$ is smooth, then $X$ cannot have a determinantal representation \cite{Dolgachev2012CAG}*{Section 4.1.1}. For $n\leq 3$ smoothness is however not an obstruction to having a determinantal representation. In this section we will show that smooth adjoints in dimension at least three exist but are rare in the sense that in each dimension $n\geq 3$ there exist only finitely many combinatorial types of polytopes for which the adjoint can be smooth. This is in contrast to the two-dimensional case: the adjoint of a generic polygon in the plane is smooth \cite{telen2025toric}*{Corollary 7.14}. Note that the existence result in dimension four implies that not every adjoint of a polytope $P$ (even if the facet hyperplane arrangement $\mathcal{H}_P$ is simple) has a determinantal representation. For a concrete counterexample, see \Cref{ex:counter}. 

We begin with treating the three-dimensional case in detail. 
More precisely, we show that for all except for finitely many combinatorial types of 3-dimensional simple convex polytopes the adjoint is necessarily singular, and that this is forced by the combinatorial structure of the residual arrangement. For the remaining finitely many combinatorial types we show that there is a representative with a smooth adjoint. Since smoothness is an open condition, this implies that in the space of all adjoints of polytopes with such a combinatorial type the smooth adjoints form a non-empty Zariski open subset. Informally speaking, for such combinatorial types there are many smooth adjoints.

\begin{lem} \label{lem:threelines}
    Let $P$ be a polytope in $\pp^3$ with a simple facet hyperplane arrangement $\mathcal{H}_P$. If the residual arrangement $\cR(P)$ contains three lines that meet at a point, then this point is a singular point of the adjoint surface of $P$.
\end{lem}

\begin{proof}
    Since $\mathcal{H}_P$ is simple, the three lines in $\cR(P)$ are not coplanar, and hence they span a three-dimensional tangent space to the adjoint hypersurface, i.e., the tangent space in a singular point.
\end{proof}

\begin{prop}\label{prop:adj_sing}
 Let $P$ be a polytope in $\pp^3$ with $k\geq9$ facets such that the facet hyperplane arrangement $\mathcal{H}_P$ is simple. Then the residual arrangement $\cR(P)$ contains three lines meeting in a single point.
 In particular, the adjoint hypersurface of $P$ is singular.
\end{prop}

\begin{proof}
 We prove that the residual arrangement $\cR(P)$ contains three lines meeting in a single point. By \Cref{lem:threelines} this also implies the second claim.

 We denote by $v$, $e$ and $k$ the number of vertices, edges and facets of $P$. Recall that the \emph{degree} of a facet of $P$ is the number of adjacent facets. We have the classical identities
 \begin{equation*}
 \begin{split}
    v - e + k = 2, \\
    \sum_F \deg(F) = 2e = 3v
 \end{split}
 \end{equation*}
 with the sum ranging over all facets of $P$. From these identities we obtain
 \begin{equation*}
     \sum_F \deg(F) = 6k - 12.
 \end{equation*}
 Next choose a facet $F_0$ of smallest degree $\delta$. From
 \begin{equation*}
    \delta k \leq 6k-12 \iff 12 \leq k(6-\delta)
 \end{equation*}
 we obtain $\delta \leq 5$. 
 We now make a case distinction based on $k$. We recall that the one-skeleton of any three-dimensional polytope is planar \cite{grunbaum1967convex}*{Section 13.1}.
 \begin{itemize}
     \item $k \geq 11$. Color $F_0$ black and all adjacent facets white. There are now at least five uncolored facets. Of these at least two are not adjacent, since otherwise the facet adjacency graph of $P$ (which is the one-skeleton of the dual polytope $P^\circ$) would contain a subgraph isomorphic to the complete graph on five vertices, which is not planar. Thus we may choose two such facets and color them black. The remaining facets are colored white. By construction, the three black facets are mutually non-adjacent. Hence they define three lines in the residual arrangement that are not coplanar and intersect in a single point.

     \item $k = 10$. Then $\delta = \deg(F_0) \leq 4$, since $12 > 10 = k \cdot (6-5)$. Now the same argument as before applies, since after coloring $F_0$ and its adjacent facets, there remain at least five uncolored facets.

     \item $k = 9$. There are two different cases that may occur. If $\delta = 3$, then we invoke the same argument again. If $\delta = 4$, then the initial coloring leaves four facets uncolored. Assume that all four of them are mutually adjacent. Then the facet adjacency graph $G$ of $P$ contains the complete graph $K_4$ on four vertices as a subgraph. Choose a planar representation of $G$. For a schematic guide through the argument see a sketch of $G$ in \Cref{fig:subK4}. The vertex corresponding to $F_0$ lies in one of the four triangular regions bounded by $K_4$. Let $S$ denote that region. By planarity, $S$ must also contain all four vertices in $G$ corresponding to facets adjacent to $F_0$ and moreover, the unique vertex of $G$ outside of $S$ is not connected to any of these four vertices nor to the vertex corresponding to $F_0$. Therefore, since $k = 9$, the unique vertex outside of $S$ has degree three, a contradiction to $\delta = 4$. \qedhere

     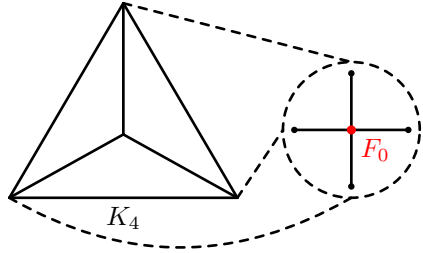
\begin{figure}[h]
         \centering
         \begin{tikzpicture}[line cap=round,line join=round,x=3cm,y=3cm]
\draw [line width=1pt] (0,0)-- (1,0);
\draw [line width=1pt] (1,0)-- (0.5,0.86);
\draw [line width=1pt] (0.5,0.86)-- (0,0);
\draw [line width=1pt] (0.5,0.86)-- (0.5,0.28);
\draw [line width=1pt] (0.5,0.28)-- (1,0);
\draw [line width=1pt] (0.5,0.28)-- (0,0);
\node[anchor=north] at (0.5,0) {$K_4$};

\draw [fill=black] (1.25,0.3) circle (1pt);
\draw [fill=black] (1.75,0.3) circle (1pt);
\draw [fill=black] (1.5,0.05) circle (1pt);
\draw [fill=black] (1.5,0.55) circle (1pt);
\draw [line width=1pt] (1.5,0.3)-- (1.25,0.3);
\draw [line width=1pt] (1.5,0.3)-- (1.75,0.3);
\draw [line width=1pt] (1.5,0.3)-- (1.5,0.05);
\draw [line width=1pt] (1.5,0.3)-- (1.5,0.55);
\node[anchor=north west, color=red] at (1.5,0.3) {$F_0$};
\draw [line width=1pt, dashed] (1.5,0.3) circle (0.9cm);
\draw [line width=1pt, dashed] (1,0)-- (1.2,0.3);
\draw [line width=1pt, dashed] (0.5,0.86)-- (1.5,0.6);
\draw [line width=1pt, dashed] (0,0) to[out=-30,in=-150] (1.5,0);
\draw [fill=black, color=red] (1.5,0.3) circle (1.5pt);

\end{tikzpicture}
         \caption{Planar representation of $G$ containing $K_4$. By symmetry, we can assume $F_0$ to lie in the unbounded region $S$ of the complement of $K_4$. The exact edges connecting $K_4$ and the remaining vertices are unknown and thus only implied as dotted lines.}
         \label{fig:subK4}
     \end{figure}
 \end{itemize}
\end{proof}

\begin{rem}
    The  bound from \Cref{prop:adj_sing} is tight. Indeed, there is a simple polytope $P$ with eight facets such that $\cR(P)$ does not contain three lines meeting in a single point, see \cite{KohnRanestad2019AdjCurves}*{Example 14} or \Cref{tab:combinatorialTypes} below.
\end{rem}
Without much additional work we obtain the following result.

\begin{cor} \label{cor:ResLines}
 Let $P$ be a polytope in $\pp^3$ with $k$ facets such that the facet hyperplane arrangement $\mathcal{H}_P$ is simple. Then the residual arrangement $\cR(P)$ contains exactly
 \begin{equation*}
    \binom{k}{2} - 3(k-2) = \binom{k-3}{2}
 \end{equation*}
 many lines.
\end{cor}

\begin{proof}
 As in the proof of \Cref{prop:adj_sing} we have the two classical identities
 \begin{equation*}
 \begin{split}
    v - e + k = 2, \\
    2e = 3v.
 \end{split}
 \end{equation*}
 Since $\mathcal{H}_P$ is simple, a line in the residual arrangement arises as the intersection of two non-adjacent facet planes. Therefore, it suffices to express $\binom{k}{2} - e$ in terms of~$k$. This computation yields the result.
\end{proof}

\begin{table}[ht]
    \centering
    \begin{tabular}{c|c|c}
        combinatorial type & residual arrangement & facet degrees \\
        \hline
        \begin{tikzpicture}[line cap=round,line join=round,x=1.5cm,y=1.5cm]
\draw [line width=1pt] (0,0)-- (1,0);
\draw [line width=1pt] (1,0)-- (0.5,0.86);
\draw [line width=1pt] (0.5,0.86)-- (0,0);
\draw [line width=1pt] (0.5,0.86)-- (0.5,0.28);
\draw [line width=1pt] (0.5,0.28)-- (1,0);
\draw [line width=1pt] (0.5,0.28)-- (0,0);
\end{tikzpicture}
& & 3333 \\
        \hline
        \begin{tikzpicture}[line cap=round,line join=round,x=1.5cm,y=1.5cm]
\draw [line width=1pt] (1.6,0)-- (2.6,0);
\draw [line width=1pt] (2.6,0)-- (2.1,0.86);
\draw [line width=1pt] (2.1,0.86)-- (1.6,0);
\draw [line width=1pt] (1.9,0.2)-- (2.3,0.2);
\draw [line width=1pt] (2.3,0.2)-- (2.1,0.54);
\draw [line width=1pt] (2.1,0.54)-- (1.9,0.2);
\draw [line width=1pt] (2.1,0.86)-- (2.1,0.54);
\draw [line width=1pt] (2.3,0.2)-- (2.6,0);
\draw [line width=1pt] (1.9,0.2)-- (1.6,0);
\end{tikzpicture}
&
\begin{tikzpicture}[line cap=round,line join=round,x=1.5cm,y=1.5cm]
\draw [line width=1pt] (1.8,0)-- (2.5,-0.8);
\draw [fill=black] (2.3,-0.3) circle (1pt);
\end{tikzpicture} & 44433 \\
        \hline
        \begin{tikzpicture}[line cap=round,line join=round,x=1.5cm,y=1.5cm]

\draw [line width=1pt] (4,0)-- (4,0.8);
\draw [line width=1pt] (4,0.8)-- (4.8,0.8);
\draw [line width=1pt] (4.8,0.8)-- (4.8,0);
\draw [line width=1pt] (4.8,0)-- (4,0);
\draw [line width=1pt] (4.2,0.2)-- (4.2,0.6);
\draw [line width=1pt] (4.2,0.6)-- (4.6,0.6);
\draw [line width=1pt] (4.6,0.6)-- (4.6,0.2);
\draw [line width=1pt] (4.6,0.2)-- (4.2,0.2);
\draw [line width=1pt] (4.2,0.2)-- (4,0);
\draw [line width=1pt] (4.2,0.6)-- (4,0.8);
\draw [line width=1pt] (4.6,0.6)-- (4.8,0.8);
\draw [line width=1pt] (4.6,0.2)-- (4.8,0);
\end{tikzpicture}
&
\begin{tikzpicture}[line cap=round,line join=round,x=1.5cm,y=1.5cm]
\draw [line width=1pt] (4.2,-0.4)-- (4,-1.2);
\draw [line width=1pt] (4.2,-1)-- (4.8,-1.2);
\draw [line width=1pt] (4.8,-1)-- (4.4,-0.4);
\end{tikzpicture} & 444444 \\
        \hline
        \begin{tikzpicture}[line cap=round,line join=round,x=1.5cm,y=1.5cm]
\draw [line width=1pt] (5.6,0)-- (6.6,0);
\draw [line width=1pt] (6.6,0)-- (6.1,0.86);
\draw [line width=1pt] (6.1,0.86)-- (5.6,0);
\draw [line width=1pt] (6,0.3)-- (6.2,0.3);
\draw [line width=1pt] (6.2,0.3)-- (6.1,0.47);
\draw [line width=1pt] (6.1,0.47)-- (6,0.3);
\draw [line width=1pt] (5.6,0)-- (5.8,0.1);
\draw [line width=1pt] (5.8,0.1)-- (6,0.3);
\draw [line width=1pt] (6.1,0.47)-- (6.1,0.86);
\draw [line width=1pt] (6.4,0.1)-- (6.2,0.3);
\draw [line width=1pt] (6.4,0.1)-- (6.6,0);
\draw [line width=1pt] (6.4,0.1)-- (5.8,0.1);
\end{tikzpicture}
&
\begin{tikzpicture}[line cap=round,line join=round,x=1.5cm,y=1.5cm]
\draw [line width=1pt] (6.4,-0.4)-- (6.2,-1.2);
\draw [line width=1pt] (6.4,-1)-- (5.6,-1);
\draw [line width=1pt] (5.6,-0.4)-- (5.8,-1.2);
\draw [fill=black] (5.8,-0.6) circle (1pt);
\draw [fill=black] (6.2,-0.6) circle (1pt);
\end{tikzpicture} & 554433 \\
        \hline
        \begin{tikzpicture}[line cap=round,line join=round,x=1.5cm,y=1.5cm]
\draw [line width=1pt] (7,0)-- (7.6,0);
\draw [line width=1pt] (7.6,0)-- (7.9,0.51);
\draw [line width=1pt] (7.9,0.51)-- (7.6,1.03);
\draw [line width=1pt] (7.6,1.03)-- (7,1.03);
\draw [line width=1pt] (7,1.03)-- (6.7,0.51);
\draw [line width=1pt] (6.7,0.51)-- (7,0);
\draw [line width=1pt] (7,0)-- (7,0.3);
\draw [line width=1pt] (7,0.3)-- (7.2,0.5);
\draw [line width=1pt] (7.2,0.5)-- (7.4,0.5);
\draw [line width=1pt] (7.4,0.5)-- (7.6,0);
\draw [line width=1pt] (7.4,0.5)-- (7.6,0.7);
\draw [line width=1pt] (7.6,0.7)-- (7.9,0.51);
\draw [line width=1pt] (7.6,0.7)-- (7.6,1.03);
\draw [line width=1pt] (7,1.03)-- (7.2,0.5);
\draw [line width=1pt] (7,0.3)-- (6.7,0.51);
\end{tikzpicture}
&
\begin{tikzpicture}[line cap=round,line join=round,x=1.5cm,y=1.5cm]
\draw [line width=1pt] (7.4,-0.2)-- (8,-0.8);
\draw [line width=1pt] (7.2,-1.4)-- (8,-0.6);
\draw [line width=1pt] (6.8,-0.8)-- (7.4,-1.4);
\draw [line width=1pt] (7,-0.6)-- (7,-1.2);
\draw [line width=1pt] (7.8,-0.4)-- (7.2,-0.4);
\draw [line width=1pt] (6.8,-1)-- (7.6,-0.2);
\draw [fill=black] (7.8,-1) circle (1pt);
\draw [fill=black] (7.6,-1.2) circle (1pt);
\end{tikzpicture} & 6554433 \\
        \hline
        \begin{tikzpicture}[line cap=round,line join=round,x=1.5cm,y=1.5cm]
\draw [line width=1pt] (15,0)-- (16,0);
\draw [line width=1pt] (16,0)-- (15.5,0.86);
\draw [line width=1pt] (15.5,0.86)-- (15,0);
\draw [line width=1pt] (15.5,0.7)-- (15.6,0.4);
\draw [line width=1pt] (15.6,0.4)-- (15.8,0.15);
\draw [line width=1pt] (15.8,0.15)-- (16,0);
\draw [line width=1pt] (15.2,0.15)-- (15,0);
\draw [line width=1pt] (15.2,0.15)-- (15.4,0.4);
\draw [line width=1pt] (15.4,0.4)-- (15.5,0.7);
\draw [line width=1pt] (15.5,0.7)-- (15.5,0.86);
\draw [line width=1pt] (15.5,0.2)-- (15.2,0.15);
\draw [line width=1pt] (15.5,0.2)-- (15.8,0.15);
\draw [line width=1pt] (15.5,0.2)-- (15.5,0.3);
\draw [line width=1pt] (15.5,0.3)-- (15.6,0.4);
\draw [line width=1pt] (15.5,0.3)-- (15.4,0.4);
\end{tikzpicture}
&
\begin{tikzpicture}[line cap=round,line join=round,x=1.5cm,y=1.5cm]
\draw [line width=1pt] (15.8,-0.2)-- (15,-1);
\draw [line width=1pt] (15,-0.8)-- (16,-1);
\draw [line width=1pt] (16,-1.2)-- (15.6,-0.2);
\draw [line width=1pt] (15.7,-0.7)-- (16,-0.5);
\draw [line width=1pt] (15.5,-1.1)-- (15.6,-0.8);
\draw [line width=1pt] (15.5,-0.7)-- (15.2,-0.5);
\draw [fill=black] (15.4,-0.4) circle (1pt);
\end{tikzpicture} & 5554443 \\
        \hline
        \begin{tikzpicture}[line cap=round,line join=round,x=1.5cm,y=1.5cm]
\draw [line width=1pt] (16.8,0)-- (17.4,0);
\draw [line width=1pt] (17.4,0)-- (17.58,0.57);
\draw [line width=1pt] (17.58,0.57)-- (17.1,0.92);
\draw [line width=1pt] (17.1,0.92)-- (16.61,0.57);
\draw [line width=1pt] (16.61,0.57)-- (16.8,0);
\draw [line width=1pt] (16.95,0.2)-- (17.25,0.2);
\draw [line width=1pt] (17.25,0.2)-- (17.34,0.48);
\draw [line width=1pt] (17.34,0.48)-- (17.1,0.66);
\draw [line width=1pt] (17.1,0.66)-- (16.85,0.48);
\draw [line width=1pt] (16.85,0.48)-- (16.95,0.2);
\draw [line width=1pt] (16.8,0)-- (16.95,0.2);
\draw [line width=1pt] (17.25,0.2)-- (17.4,0);
\draw [line width=1pt] (17.34,0.48)-- (17.58,0.57);
\draw [line width=1pt] (17.1,0.66)-- (17.1,0.92);
\draw [line width=1pt] (16.85,0.48)-- (16.61,0.57);
\end{tikzpicture}
&
\begin{tikzpicture}[line cap=round,line join=round,x=1.5cm,y=1.5cm]
\draw [line width=1pt] (17.25,-0.2)-- (16.6,-0.5);
\draw [line width=1pt] (16.7,-0.3)-- (16.8,-1.1);
\draw [line width=1pt] (16.7,-1)-- (17.5,-1);
\draw [line width=1pt] (17.5,-0.3)-- (17.4,-1.1);
\draw [line width=1pt] (16.95,-0.2)-- (17.6,-0.5);
\draw [line width=1pt] (17.8,-1)-- (17.8,-0.3);
\end{tikzpicture} & 5544444 \\
        \hline
        \begin{tikzpicture}[line cap=round,line join=round 4,x=0.8cm,y=0.8cm]
\draw [line width=1pt] (0,0)-- (1,0);
\draw [line width=1pt] (1,0)-- (1.5,0.86);
\draw [line width=1pt] (1.5,0.86)-- (1,1.73);
\draw [line width=1pt] (1,1.73)-- (0,1.73);
\draw [line width=1pt] (0,1.73)-- (-0.5,0.86);
\draw [line width=1pt] (-0.5,0.86)-- (0,0);
\draw [line width=1pt] (0,0)-- (0,0.4);
\draw [line width=1pt] (0,0.4)-- (0.2,0.8);
\draw [line width=1pt] (0.2,0.8)-- (0.4,1);
\draw [line width=1pt] (0.4,1)-- (0.8,0.8);
\draw [line width=1pt] (0.8,0.8)-- (1,0);
\draw [line width=1pt] (0.8,0.8)-- (1,1);
\draw [line width=1pt] (1,1)-- (1.5,0.86);
\draw [line width=1pt] (1,1)-- (1,1.73);
\draw [line width=1pt] (0,1.73)-- (0.4,1);
\draw [line width=1pt] (-0.5,0.86)-- (-0.2,0.8);
\draw [line width=1pt] (-0.2,0.8)-- (0,0.4);
\draw [line width=1pt] (-0.2,0.8)-- (0.2,0.8);
\end{tikzpicture}
&
\begin{tikzpicture}[line cap=round,line join=round 4,x=0.3cm,y=0.3cm]
\draw [line width=1pt] (-2,0)-- (5,0);
\draw [line width=1pt] (-2,3)-- (5,3);
\draw [line width=1pt] (-1,4)-- (-1,-1);
\draw [line width=1pt] (4,4)-- (4,-1);
\draw [line width=1pt] (-1.5,2)-- (1.5,3.5);
\draw [line width=1pt] (-1.5,1)-- (1.5,-0.5);
\draw [line width=1pt] (4.5,2)-- (2.5,3.5);
\draw [line width=1pt] (4.5,1)-- (2.5,-0.5);
\draw [line width=1pt] (-1.5,-0.75)-- (3.75,2.75);
\draw [line width=1pt] (-1.5,3.75)-- (3.75,0.25);

\draw [fill=black] (5,2) circle (1pt);
\draw [fill=black] (5,1) circle (1pt);
\draw [fill=white] (1.9,1.5) circle (1.5pt);
\end{tikzpicture} & 65554443 \\
        \hline
        \begin{tikzpicture}[line cap=round,line join=round,x=2cm,y=2cm]

\draw [line width=1pt] (0,0)-- (1,0);
\draw [line width=1pt] (1,0)-- (0.5,0.86);
\draw [line width=1pt] (0.5,0.86)-- (0,0);
\draw [line width=1pt] (0.5,0.86)-- (0.5,0.7);
\draw [line width=1pt] (0.5,0.7)-- (0.6,0.4);
\draw [line width=1pt] (1,0)-- (0.8,0.15);
\draw [line width=1pt] (0.8,0.15)-- (0.6,0.4);
\draw [line width=1pt] (0,0)-- (0.2,0.15);
\draw [line width=1pt] (0.2,0.15)-- (0.5,0.2);
\draw [line width=1pt] (0.5,0.2)-- (0.8,0.15);
\draw [line width=1pt] (0.2,0.15)-- (0.4,0.4);
\draw [line width=1pt] (0.4,0.4)-- (0.5,0.7);
\draw [line width=1pt] (0.4,0.4)-- (0.45,0.35);
\draw [line width=1pt] (0.5,0.2)-- (0.5,0.25);
\draw [line width=1pt] (0.45,0.35)-- (0.5,0.25);
\draw [line width=1pt] (0.5,0.25)-- (0.55,0.35);
\draw [line width=1pt] (0.55,0.35)-- (0.45,0.35);
\draw [line width=1pt] (0.55,0.35)-- (0.6,0.4);
\end{tikzpicture}
&
\begin{tikzpicture}[line cap=round,line join=round,x=1.5cm,y=1.5cm]
\draw [line width=1pt] (0,-0.2)-- (0,-1.4);
\draw [line width=1pt] (0.8,-0.4)-- (-0.2,-0.4);
\draw [line width=1pt] (0.5,-0.3)-- (-0.2,-0.7);
\draw [line width=1pt] (0.5,-0.7)-- (-0.1,-0.2);
\draw [line width=1pt] (0.8,-1.2)-- (-0.2,-1.2);
\draw [line width=1pt] (0.5,-0.9)-- (-0.1,-1.4);
\draw [line width=1pt] (-0.2,-0.9)-- (0.5,-1.3);
\draw [line width=1pt] (0.7,-1.3)-- (0.7,-0.3);
\draw [line width=1pt] (0.4,-0.5)-- (0.4,-1.1);
\draw [line width=1pt] (-0.1,-0.5)-- (-0.1,-1.1);

\draw [fill=black] (0.8,-1) circle (1pt);
\draw [fill=black] (0.8,-0.6) circle (1pt);
\end{tikzpicture} & 55555533 \\
        \hline
        \begin{tikzpicture}[line cap=round,line join=round,x=1.8cm,y=1.8cm]

\draw [line width=1pt] (13,0)-- (13.6,0);
\draw [line width=1pt] (13.6,0)-- (13.78,0.57);
\draw [line width=1pt] (13.78,0.57)-- (13.3,0.92);
\draw [line width=1pt] (13.3,0.92)-- (12.81,0.57);
\draw [line width=1pt] (12.81,0.57)-- (13,0);
\draw [line width=1pt] (13.15,0.25)-- (13.45,0.25);
\draw [line width=1pt] (13.45,0.25)-- (13.54,0.53);
\draw [line width=1pt] (13.54,0.53)-- (13.3,0.71);
\draw [line width=1pt] (13.3,0.71)-- (13.05,0.53);
\draw [line width=1pt] (13.05,0.53)-- (13.15,0.25);
\draw [line width=1pt] (13,0)-- (13.1,0.1);
\draw [line width=1pt] (13.1,0.1)-- (13.15,0.25);
\draw [line width=1pt] (13.1,0.1)-- (13.5,0.1);
\draw [line width=1pt] (13.5,0.1)-- (13.45,0.25);
\draw [line width=1pt] (13.5,0.1)-- (13.6,0);
\draw [line width=1pt] (13.54,0.53)-- (13.78,0.57);
\draw [line width=1pt] (13.3,0.71)-- (13.3,0.92);
\draw [line width=1pt] (13.05,0.53)-- (12.81,0.57);
\end{tikzpicture}
&
\begin{tikzpicture}[line cap=round,line join=round,x=1.5cm,y=1.5cm]
\draw [line width=1pt] (13.05,-1.15)-- (13.75,-0.45);
\draw [line width=1pt] (13.5,-1.3)-- (13.9,-0.9);
\draw [line width=1pt] (12.9,-0.9)-- (13.3,-1.3);
\draw [line width=1pt] (13,-1.4)-- (13,-0.2);
\draw [line width=1pt] (14,-0.4)-- (12.8,-0.4);
\draw [line width=1pt] (13.3,-0.3)-- (12.9,-0.7);
\draw [line width=1pt] (13.75,-1.15)-- (13.05,-0.45);
\draw [line width=1pt] (13.9,-0.7)-- (13.5,-0.3);
\draw [line width=1pt] (14,-1.2)-- (12.8,-1.2);
\draw [line width=1pt] (13.8,-1.4)-- (13.8,-0.2);
\draw [line width=1pt] (13.05,-1.15)-- (13.75,-0.45);
\draw [line width=1pt] (13.75,-1.15)-- (13.05,-0.45);

\draw [fill=white] (13.4,-0.8) circle (1.5pt);
\end{tikzpicture} & 55554444 \\
        \hline
    \end{tabular}
    \caption{All combinatorial types (in stereographic projection) of simple convex polytopes in $\pp^3$, for which a representative with a smooth adjoint hypersurface exists,
    and their residual arrangements. The symbol $\circ$ indicates that two lines do not intersect.}
    \label{tab:combinatorialTypes}
\end{table}

We are now ready to give a characterization of three-dimensional polytopes with smooth adjoints. 
\begin{thm}\label{thm:smooth3d}
    There are exactly ten combinatorial types of simple polytopes in $\mathbb{P}^3$ that have a representative with a smooth adjoint. These are shown along with their residual arrangements in \Cref{tab:combinatorialTypes}.
\end{thm}

\begin{proof}
By \Cref{prop:adj_sing} any three-dimensional polytope with at least nine facets has a singular adjoint. An explicit list of (polar duals of) three-dimensional simple polytopes with at most eight facets is given in \cite{BrittonDunitz73Polytopes} (note that we are only interested in the triangulated graphs presented there). In the list of polytopes with eight facets, only those with labels 7, 10, and 14 do not have three residual lines meeting in a point. In the list of polytopes with seven facets, only those with labels 3, 4, and 5 do not have this property. All four combinatorial types of polytopes with at most six facets do not have it as well. \Cref{tab:combinatorialTypes} shows all these combinatorial types in stereographic projection along with their residual arrangements. 
For polytopes with four or five facets the adjoint is either a constant or a linear polynomial. 
For the remaining eight combinatorial types we constructed explicit representatives, computed their adjoint polynomials and verified that these define smooth hypersurfaces in \texttt{Macaulay2}. Our code is available at \cite{detrepscode}.
\end{proof}

We now turn to investigating smoothness of adjoints in dimension at least four. We begin with  a concrete example of a smooth adjoint of a four-dimensional polytope with a simple facet hyperplane arrangement.

\begin{ex} \label{ex:counter}
    Consider the polytope in $\mathbb{P}^4$ defined by the facet inequalities
    \begin{multline*}
        x_0\geq 0,\ x_1\geq 0,\ x_2\geq 0, x_3\geq 0, x_4\geq 0,\\ 2x_0+3x_1+x_2-x_3-3x_4\geq 0,-2x_0-x_1+x_2+2x_3+2x_4\geq 0.
    \end{multline*}
    The residual arrangement of this polytope consists of seven lines. In particular, it has no planes in it. The adjoint is the quadratic threefold given by
    $$2x_0x_2 + 3x_1x_2 + x_2^2  + 2x_0x_3 + 5x_1x_3 + 2x_2x_3 + 3x_1x_4 + 2x_2x_4=0.$$
    This is a smooth hypersurface in $\mathbb{P}^4$ and thus, by \Cref{prop:quadratic}, its defining polynomial does not admit a determinantal representation. It also follows from \Cref{prop:quadratic} that it does not contain a plane.
\end{ex}

We obtained this example by analyzing quadratic adjoints of different combinatorial types of simple four- and five-dimensional polytopes. There are five such combinatorial types in four dimensions, and eight in five dimensions. They are indexed by Gale diagrams of their polar dual polytopes in \cite{grunbaum1967convex}*{Figure 6.3.3}. This analysis leads us to the following statement, the proof of which is computational. 

\begin{prop}\label{prop:45quadrics}
Let $P$ be a four- or five-dimensional polytope such that the facet hyperplane arrangement $\mathcal{H}_P$ is simple and the adjoint polynomial $\alpha_P$ is of degree two. If the residual arrangement $\cR(P)$ contains a linear space of codimension two, then $\alpha_P$ has a determinantal representation. Otherwise, in the set of  adjoints of polytopes with the same combinatorial type as $P$ there exists a non-empty Zariski-open set of polytopes whose adjoint is smooth and thus does not have a determinantal representation. 
\end{prop}

\begin{proof}
    If the adjoint hypersurface $A_P$ contains an intersection of two facet hyperplanes, then $\alpha_P$ has a determinantal representation by \Cref{prop:quadratic}. Now suppose $P$ does not contain an intersection of two facet hyperplanes. There are finitely many possibilities for the combinatorial type of $P$: one in dimension four and two in dimension five (in \cite{grunbaum1967convex}*{Figure 6.3.3} they correspond to the last entry for $d=4$ and the last two entries for $d=5$). For each of these three types, we explicitly construct a polytope with a smooth adjoint. Since smoothness is a Zariski-open condition, the claim follows. The code containing the analysis of these combinatorial types is available at \cite{detrepscode}.
\end{proof}

\begin{rem}
    The meaning of \Cref{prop:45quadrics} is that, typically, the quadratic adjoint of a four- or five-dimensional polytope does not contain a linear subspace of codimension two, unless the residual arrangement contains such a linear subspace. The latter is a purely combinatorial condition.    
\end{rem}

Our final goal in this section is to show that in each dimension, starting from three, there exist only finitely many combinatorial types of polytopes for which the adjoint can be smooth. This will be achieved in \Cref{thm:smoothness}. We first prove the claim for polytopes whose facet hyperplane arrangement is simple, and then extend the result by degeneration. We start our way towards it with an auxiliary definition. 

\begin{Def}[Residual $m$-plane]
    Let $P$ be a polytope of dimension $n$ such that $\mathcal{H}_P$ is simple. By a \emph{residual $m$-plane} of $P$ we mean an $m$-plane contained in the residual arrangement $\mathcal{R}(P)$ that arises as the intersection of $n-m$ hyperplanes in $\mathcal{H}_P$. Note that we do not require this $m$-plane to be an irreducible component of $\mathcal{R}(P)$: it can be contained in a higher-dimensional component thereof.
\end{Def}

The following is an extension of \Cref{lem:threelines} to arbitrary dimensions.
\begin{lem}\label{lem:nlines_sing}
    Let $P$ be a polytope of dimension $n$ such that $\mathcal{H}_P$ is simple. If there exist $n$ residual lines in $\mathcal{R}(P)$ meeting in a point, then this point is a singular point of the adjoint hypersurface of $P$.
\end{lem}

\begin{proof}
    Since the adjoint hypersurface contains $\mathcal{R}(P)$, it contains the $n$ residual lines. Since $\mathcal{H}_P$ is a simple arrangement, these lines span an $n$-dimensional tangent space to the adjoint at their intersection point, yielding a singularity.
\end{proof}

In the following theorem we make use of the big $O$ notation: for two functions $f,g\colon\mathbb{N}\to\mathbb{N}$ we say that $f=O(g)$ if there exists $C\in\mathbb{N}$ such that $f(x)<Cg(x)$ for all $x\in\mathbb{N}$. 

\begin{thm}  \label{thm:smoothness}
    For all $n\geq3$ there exists $N\in\N$ such that if a polytope $P$ of dimension $n$ has $k>N$ facets, its adjoint is a singular hypersurface.
\end{thm}

\begin{proof}
    We first prove the claim in the case that $\cH_P$ is simple.
    For a simple $n$-dimensional polytope $P$ with $k$ facets the number of faces of dimension $i$ grows at most as $O(k^{n-i})$ for $i> \lfloor \frac{n}{2} \rfloor$ and at most as $O(k^{\lfloor \frac{n}{2} \rfloor})$ for $i\leq \lfloor \frac{n}{2} \rfloor$, see \cite{barvinok}*{Theorem VI.7.4}.
    Since $\mathcal{H}_P$ is simple, the number of zero-dimensional intersections of hyperplanes in $\mathcal{H}_P$ is $\binom{k}{n}$. Since at most $O(k^{\lfloor \frac{n}{2} \rfloor})$ of them are vertices of $P$, the number of residual points is $\binom{k}{n}-O(k^{\lfloor \frac{n}{2} \rfloor})$ which grows like $k^n$. Now, since $n\geq 3$, the number of edges of $P$ is at most $O(k^{\lfloor \frac{n}{2} \rfloor})$. Each edge is contained in the intersection of $n-1$ hyperplanes in $\mathcal{H}_P$. Intersecting with two of the remaining $k-n+1$ hyperplanes yields two vertices of $P$ on this edge. The number of residual points on the line spanned by this edge is thus $k-n-1$. Therefore, the number of residual points that lie on at least one line containing an edge of $P$ is at most $O(k^{\lfloor\frac{n}{2}\rfloor+1})$. Since $\lfloor\frac{n}{2}\rfloor+1 < n$ for $n\geq 3$, there exists $N\in\mathbb{N}$ such that for $k>N$ there is a residual point that is not contained in the line spanned by any edge of $P$. It remains to note that every residual point, being an intersection of $n$ hyperplanes in $\mathcal{H}_P$, is also an intersection of $n$ lines, each of which is either residual or contains an edge of $P$. In particular, any residual point that is not contained in the line spanned by any edge of $P$, is contained in $n$ residual lines. By \Cref{lem:nlines_sing} any such vertex is a singular point of the adjoint. This proves the claim for $\cH_P$ simple.

    If $P$ is an arbitrary $n$-dimensional polytope with $k>N$ facets, then it is the limit of a sequence of $n$-dimensional polytopes $(P_i)_{i\in\N}$ with $k$ facets whose facet hyperplane arrangement $\cH_P$ is simple. The adjoint of each $P_i$ is singular and their limit is the adjoint of $P$ by \cite{KohnRanestad2019AdjCurves}*{Corollary 1}. Because being singular is a closed condition, it follows that the adjoint of $P$ is singular.
\end{proof}

In \Cref{ex:counter} and \Cref{prop:45quadrics} we saw that in dimension at least four adjoint hypersurfaces of polytopes do not necessarily have a determinantal representation. In the following two sections we will therefore concentrate on dimensions two and three. In our examples, the obstruction for a determinantal representation was the adjoint hypersurface being smooth. This leads to the following question.

\begin{frag}
Do adjoint hypersurfaces of polytopes that have a sufficiently high-dimensional singular locus have determinantal representations?
\end{frag}

\section{Two Dimensions} \label{sec:3}
In this section we study determinantal representations of adjoint curves of convex polygons. Since any projective polygon lives in some affine chart of~$\mathbb{P}^2$, in what follows we work in affine space. A determinantal representation of the adjoint in~$\mathbb{P}^2$ can be obtained from one in the chosen affine chart by homogenizing the entries of the matrix. The following is the main result of this section.

\begin{thm} \label{thm:recursiveDetRep}
 Let $P$ be a convex polygon with $n$ vertices $v_1, \dots, v_n$. There exists a definite symmetric determinantal representation $M$ of $\alpha_P$ satisfying the following:
 \begin{enumerate}
     \item $M$ is \emph{tridiagonal}, i.e., if $M_{ij} \neq 0$, then $|i-j| \leq 1$.
     \item The $k$-th leading submatrix of $M$ is a determinantal representation of the adjoint of the convex hull of $v_1, \dots, v_{k+3}$.
 \end{enumerate}
\end{thm}
We first set up some notation.  In what follows, we interpret the indices modulo $n$: if $i\not\in\{1,\ldots,n\}$, then we replace it with the congruent element in $\{1,\ldots,n\}$. 
Let $P \subseteq \R^2$ (with coordinates $x_1, x_2$) be a convex polygon with $n$ vertices and edges. Let $V(P) = \{v_i\}_{i=1}^n$ denote its set of vertices (ordered counterclockwise), and let $E(P) = \{e_i\}_{i=1}^n$ denote its set of edges. Here $e_i$ is the edge between $v_{i-1}$ and $v_i$. Furthermore, we define linear polynomials $l_i = \langle w_i, x \rangle + c_i$ such that $e_i \subseteq L_i := \cV(l_i)$ for all $i$ and such that $P$ is the intersection of the half-planes defined by $l_i \geq 0$. We say that a polygon $P'$ is a \emph{subpolygon} of $P$, if $V(P') \subseteq V(P)$. 
To have a convenient formula for the adjoint polynomial $\alpha_P$ at hand, we now state \cite{PavlovTelen2024PosGeom}*{Proposition 2.1} for the case of polygons in the plane.

\begin{lem}
    In the notation above, the adjoint polynomial of the convex polygon $P$ is given by
    \begin{equation} \label{eq:adjpol_expl}
        \alpha_P(x) = \sum_{i = 1}^n \det(w_i, w_{i+1}) \prod_{j \not\in\{ i, i+1\}} l_j.
    \end{equation}
\end{lem}

It is clear from this description that if we replace $l_i$ by $\gamma l_i$ with $\gamma >0$, then $\alpha_P$ changes exactly by the factor $\gamma$. This reflects the fact that the adjoint polynomial of $P$ is only defined up to multiplication by a scalar, cf. \Cref{sec:2}. In what follows, if we claim that a determinantal representation of $\alpha_P$ is defined over a certain field, we mean that one can choose a scalar multiple of $\alpha_P$ in a way that this is true.

Proving \Cref{thm:recursiveDetRep} requires a good understanding of the intersection of the adjoint curve of $P$, the edges of $P$, and the adjoint curves of certain subpolygons. We will prove several results pertaining to this. The first one identifies certain tangent lines of the adjoint of $P$ as the adjoint curves of sub-quadrilaterals of $P$.
\begin{lem} \label{lem:AdjTangent}
    Let $i,j\in\{1,\ldots,n\}$ such that $\{q\} = L_i \cap L_j\subseteq\cR(P)$, and let $Q$ denote the quadrilateral spanned by the vertices $v_{i-1}, v_i, v_{j-1}, v_j$. Then $\cV(\alpha_{Q})$ is the tangent line to $\cV(\alpha_P)$ at $q$.
\end{lem}

\begin{proof}
    We first note that, since $q$ is a real point, the adjoint curve $\cV(\alpha_P)$ is smooth in $q$ by \cite{KohnEtAl2024AdjCurves}*{Theorem 3.8}. Thus the notion of a tangent line is well-defined. We then proceed by recalling that the Zariski closure in the affine plane of the edge $e_i$ of $P$ is the vanishing locus of the linear form $l_i = \langle w_i, x \rangle + c_i$. We also recall that, by construction, $P$ is contained in the half space defined by $l_i \geq 0$.
    Next we note that, without loss of generality, we can make the following assumptions:
    \begin{enumerate}
        \item $i = 2$ and $2(j-1) \leq n$ after a cyclic permutation of the labels and exchanging $i$ and $j$ if necessary;
        \item $q = 0$ is the origin after a translation;
        \item $v_2 = (0,1), v_{j-1} = (1,0)$ after a change of basis;
        \item $l_2 = x_1, l_j = x_2$ after scaling by a positive real number.
    \end{enumerate}
    In what follows, we denote by $l_\alpha$ the linear term in $\alpha_P$. There are two cases to distinguish, either $j = 4 = i + 2$ or $j > 4$. 
    \begin{enumerate}
    \item[\textbf{Case 1}:] $j = 4$. In this case $Q$ is given by four consecutive vertices $v_1,v_2,v_3,v_4$ of $P$. In particular, three of the edges of $Q$ are edges of $P$, and the remaining edge is a diagonal of $P$.
    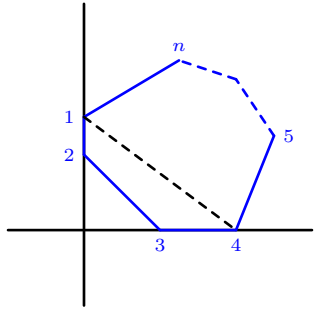
\begin{figure}[h]
\begin{tikzpicture}[line cap=round,line join=round,x=1cm,y=1cm]
\draw [line width=1pt] (-1,0)-- (3,0);
\draw [line width=1pt] (0,-1)-- (0,3);
\draw [line width=1pt, color=blue] (1.25,2.25)-- (0,1.5);
\node[anchor=east, color=blue] at (0,1.5) {\scriptsize1};
\draw [line width=1pt, color=blue] (0,1.5)-- (0,1);
\node[anchor=east, color=blue] at (0,1) {\scriptsize2};
\draw [line width=1pt, color=blue] (0,1)-- (1,0);
\node[anchor=north, color=blue] at (1,0) {\scriptsize3};
\draw [line width=1pt, color=blue] (1,0)-- (2,0);
\node[anchor=north, color=blue] at (2,0) {\scriptsize4};
\draw [line width=1pt, color=blue] (2,0)-- (2.5,1.25);
\node[anchor=west, color=blue] at (2.5,1.25) {\scriptsize5};
\draw [line width=1pt, color=blue, dashed] (2.5,1.25)-- (2,2);
\draw [line width=1pt, color=blue, dashed] (2,2)-- (1.25,2.25);
\node[anchor=south, color=blue] at (1.25,2.25) {\scriptsize{$n$}};
\draw [line width=1pt, dashed] (0,1.5)-- (2,0);
\end{tikzpicture}
\caption{Location of vertices and edges after change of coordinates in case 1. The dashed blue line indicates several vertices omitted in the figure. The dashed black line cuts off a subquadrilateral of $P$.}
    \label{fig:case1}
    \end{figure}
    In this case we have $l_3 = x_1+x_2 - 1$ (up to scaling). Note that $l_2$ and $l_4$ are both of pure degree one, and hence to compute $l_\alpha$ we may ignore all summands in \Cref{eq:adjpol_expl} that do not correspond to one of the vertices $v_{1},v_2,v_{3}, v_4$. Let $\sigma = \prod_s c_s \neq 0$ where the product is over all indices corresponding to edges that are neither equal nor adjacent to the edges $e_2$ and $e_4$. Writing $w_m = (w_m^1, w_m^2) \in \R^2$ and using the four assumptions above to simplify \Cref{eq:adjpol_expl} we may explicitly read off $l_\alpha$:
    \begin{equation} \label{eq:case1_LinTermRed}
        \begin{split}
        l_\alpha &= w_1^2 c_5 \sigma x_2 + c_1c_5 \sigma x_2 + c_1c_5 \sigma x_1 + w_5^1 c_1 \sigma x_1 \\
        &= \sigma c_5 (c_1 + w_1^2) x_2 + \sigma c_1 (c_5 + w_5^1)x_1.
        \end{split}
    \end{equation}
    Next we consider $\alpha_{Q}$. Let $l = \langle w, x \rangle + c$ denote the line connecting $v_4$ and $v_1$, such that $Q$ is given by the conditions $l \geq 0, l_i \geq 0, i = 2,3,4$. From this, we may compute $\alpha_{Q}$ in an analogous fashion as $l_\alpha$:
    \begin{equation} \label{eq:case1_tRed}
        \alpha_{Q} = (c + w^2)x_2 + (c + w^1)x_1.
    \end{equation}
    Since the line defined by $l$ has a common intersection with those defined by $l_1$ and $l_2$, we may write $\gamma_1 l = l_1 - l_2$ with $\gamma_1 > 0$ after possibly subsuming a scaling factor in $l_1$. Similarly, we have $\gamma_5 l = - l_4 + l_5$ with $\gamma_5 > 0$. Resolving these identities using our four assumptions yields
    \begin{equation*}
        c_1 = \gamma_1 c, \qquad c_5 = \gamma_5 c, \qquad w = (\gamma_5^{-1} w_5^1, \gamma_1^{-1} w_1 ^2).
    \end{equation*}
    In particular, we obtain
    \begin{equation*}
        l_\alpha = \sigma \gamma_1 \gamma_5 ((c + w^2)x_2 + (c + w^1)x_1) = \sigma \gamma_1 \gamma_5 \alpha_{Q}
    \end{equation*}
    from \Cref{eq:case1_LinTermRed} and \Cref{eq:case1_tRed}. This proves tangency.

        \begin{figure}[ht]
\begin{tikzpicture}[line cap=round,line join=round,x=1cm,y=1cm]
\draw [line width=1pt] (-1,0)-- (3,0);
\draw [line width=1pt] (0,-1)-- (0,3);
\draw [line width=1pt, color=blue] (1.25,2.25)-- (0,1.5);
\node[anchor=east, color=blue] at (0,1.5) {\scriptsize1};
\draw [line width=1pt, color=blue] (0,1.5)-- (0,1);
\node[anchor=east, color=blue] at (0,1) {\scriptsize2};
\draw [line width=1pt, color=blue] (0,1)-- (0.25,0.5);
\draw [line width=1pt, color=blue, dashed] (0.25,0.5)-- (0.65,0.15);
\draw [line width=1pt, color=blue] (0.65,0.15)-- (1,0);
\node[anchor=north, color=blue] at (1,0) {\scriptsize{$j-1$}};
\draw [line width=1pt, color=blue] (1,0)-- (2,0);
\node[anchor=north, color=blue] at (2,0) {\scriptsize{$j$}};
\draw [line width=1pt, color=blue] (2,0)-- (2.5,1.25);
\node[anchor=west, color=blue] at (2.5,1.25) {\scriptsize{$j+1$}};
\draw [line width=1pt, color=blue, dashed] (2.5,1.25)-- (2,2);
\draw [line width=1pt, color=blue, dashed] (2,2)-- (1.25,2.25);
\node[anchor=south, color=blue] at (1.25,2.25) {\scriptsize{$n$}};
\draw [line width=1pt, dashed] (0,1)-- (1,0);
\draw [line width=1pt, dashed] (0,1.5)-- (2,0);
\end{tikzpicture}
    \caption{Location of vertices and edges after change of coordinates in case 2. The dashed blue lines indicate several vertices omitted in the figure. The dashed black lines cut out a subquadrilateral of $P$.}
    \label{fig:case2}
    \end{figure}
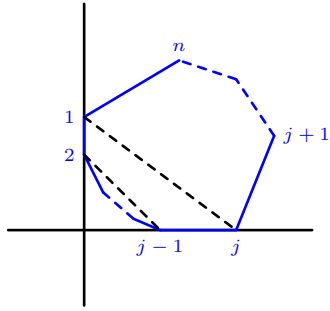    
    \item[\textbf{Case 2}:] $j > 4$. Then the edges $e_2$ and $e_j$ are non-adjacent, and two of the sides of $Q$ are diagonals of $P$.
    We begin by computing $l_\alpha$. As before, we may ignore terms that do not correspond the vertices $v_{1},v_2,v_{j-1}v_j$. We define $\sigma$ in an analogous way as before, which yields
    \begin{equation} \label{eq:case2_LinTermRed}
    \begin{split}
        l_\alpha &= -w_1^2 c_3 c_{j-1} c_{j+1} \sigma x_2 + w_3^2 c_1 c_{j-1} c_{j+1} \sigma x_2 \\
        &+ w_{j-1}^1 c_1 c_3 c_{j+1} \sigma x_1 - w_{j+1}^1 c_1 c_3 c_{j-1} \sigma x_1 \\
        &= \sigma c_{j-1}c_{j+1}(c_1w_3^2 - c_3w_1^2)x_2 + \sigma c_1c_3 (c_{j+1} w_{j-1}^1 - c_{j-1} w_{j+1}^1)x_1.
    \end{split}
    \end{equation}    
    Considering $Q$ next, we observe that there are two edges in $Q$ that were not already edges of $P$, namely one connecting the vertices $v_2, v_{j-1}$ (called~$l'$) and one connecting the vertices $v_1,v_j$ (called $l$). After scaling we write $l' = x_1+x_2 - 1$, and we write $l = \langle w, x \rangle + c$ for variables $w, c$. The polynomial $\alpha_{Q}$ is then given by
    \begin{equation} \label{eq:case2_tRed}
        \alpha_{Q} = (c + w^2) x_2 + (c + w^1) x_1.
    \end{equation}    
    As in the first case we have
    \begin{align*}
        \gamma_3 l' = - l_2 + l_3, &\qquad \gamma_{j-1} l' = l_{j-1} - l_j \\
        \gamma_1 l = l_1 - l_2, &\qquad \gamma_{j+1} l = -l_j + l_{j+1}
    \end{align*}
    where we always subsume scalar factors in the linear forms $l_i$. As a consequence we obtain
    \begin{equation*}
    \begin{split}
        w_1^2 = \gamma_1w^2, &\qquad c_1 = \gamma_1 c \\
        w_3^2 = \gamma_3, &\qquad c_3 = -\gamma_3 \\
        w_{j-1}^1 = \gamma_{j-1}, &\qquad c_{j-1} = -\gamma_{j-1} \\
        w_{j+1}^1 = \gamma_{j+1}w^1, &\qquad c_{j+1} = \gamma_{j+1} c
    \end{split}
    \end{equation*}
    Substituting these in \Cref{eq:case2_LinTermRed}, writing $\gamma = \gamma_1\gamma_3\gamma_{j-1}\gamma_{j+1}$, and comparing with \Cref{eq:case2_tRed} yields
    \begin{equation*}
        l_\alpha = - \sigma \gamma c (c + w^2)x_2 - \sigma \gamma c (c + w^1)x_1 = - \sigma \gamma c \alpha_{Q}.
    \end{equation*}
    Once again, this proves tangency.\qedhere
    \end{enumerate}
\end{proof}

\begin{prop}\label{thm:adjointcontact}
    Let $P$ be a convex polygon with $n$ vertices and adjoint curve $A = \cV(\alpha_P)$. Let $P'$ be the polygon obtained from $P$ by removing the vertex $v_n$. The intersection of the adjoint curve $A'$ of $P'$ with $A$ is equal to $\cR(P) \smallsetminus (L_1 \cup L_n)$. Moreover, each intersection point is of multiplicity two.
\end{prop}

\begin{proof}
    Since all (non-adjacent) edges of $P$, except $e_1$ and $e_n$, are also (non-adjacent) edges of $P'$, we observe that $\cR(P)\smallsetminus ( L_1 \cup L_n) \subseteq \cR(P')$. Furthermore, we calculate
\begin{equation*}
    2 \cdot| \cR(P)\smallsetminus ( L_1 \cup L_n) |     = (n-3)(n-4) = \deg(A)\cdot\det(A').
\end{equation*}
By B\'ezout's theorem it thus suffices to prove that in every point $q \in \cR(P) \smallsetminus (L_1 \cup L_n)$ the curves $A$ and $A'$ have the same tangent line. For that let $\{q\} = L_i \cap L_j$. By assumption $i,j \not \in \{1, n\}$. Hence, \Cref{lem:AdjTangent} may be applied to both $P$ and $P'$, which proves the desired tangency.
\end{proof}
Now we are almost ready to prove \Cref{thm:recursiveDetRep}. We start with a proposition.
\begin{prop} \label{prop:quadrilrest}
 Let $P$ be a polygon with $n$ vertices, let $P_k$ be the convex hull of the vertices $v_1, \dots, v_k$ and let $Q = \overline{P \smallsetminus P_{n-2}}$ be the quadrilateral that is the convex hull of the vertices $v_1,v_{n-2},v_{n-1},v_n$. We write $A_k$ for the adjoint curve of $P_k$ and $A_Q$ for that of $Q$ (which is a line). Then
 \begin{equation*}
     A_n.A_Q = A_{n-2}.A_Q + 2L_{n-1}.A_Q.
 \end{equation*}
 If  $A_n$ is smooth, then  $A_{n-2} \cup A_Q$ is a contact curve to $A_n$ with contact divisor
 \begin{equation*}
     \frac{1}{2}(A_{n-2}\cup A_Q).A_n = (A_{n-2}\cup L_{n-1}).A_n - \frac{1}{2}A_{n-1}.A_n.
 \end{equation*}
\end{prop}

\begin{figure}[h]
    \centering
    \includegraphics[width=\linewidth]{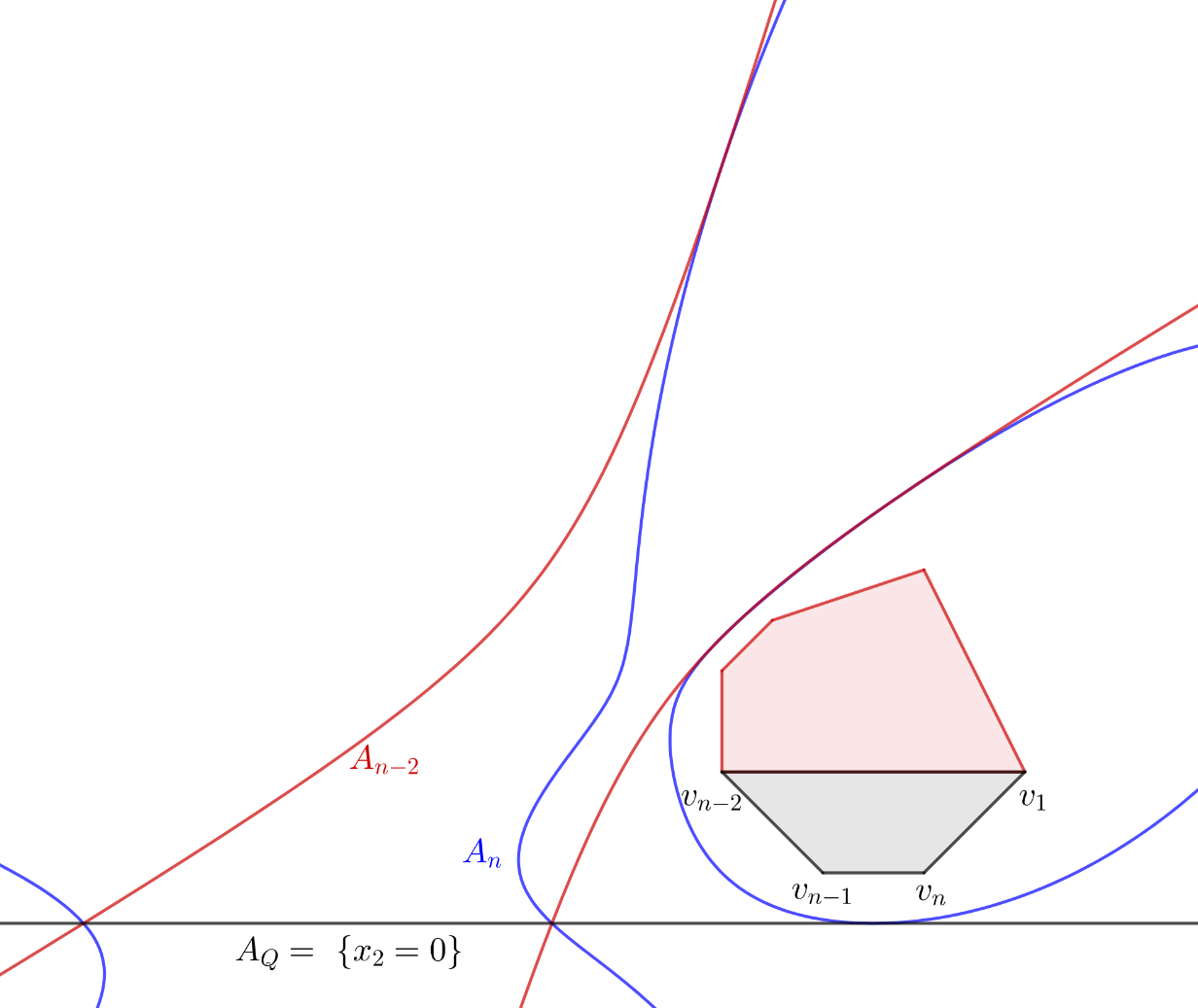}
    \caption{The blue line is the adjoint (quartic) curve of the heptagon. The red and black lines are the adjont curves of the correspondingly colored subpolygons.}
    \label{fig:WLOGsubquadril}
\end{figure}

\begin{proof}
 After choosing a suitable affine chart in $\pp^2$ and applying a linear coordinate change we may without loss of generality assume that
 \begin{equation*}
     v_{n-2} = (-\gamma,\gamma), v_{n-1} = (-1,1), v_n = (1,1), v_1 = (\gamma,\gamma)
 \end{equation*}
 for some $\gamma > 1$ (see \Cref{fig:WLOGsubquadril}). Then $A_Q = \cV(x_2)$ is the $x_1$-axis. Writing $\alpha_k$ for the adjoint polynomial of $A_k$ as defined in \Cref{eq:adjpol_expl}, it suffices to prove that
 \begin{equation*}
     -\gamma\alpha_n(x_1,0) = x_1^2\alpha_{n-2}(x_1,0).
 \end{equation*}
 Recall that $l_i(x_1,x_2) = \langle w_i, \binom{x_1}{x_2} \rangle + c_i \geq 0$ are the facet inequalities of $P$. For the sake of readability we write $w_i = (a_i,b_i)^t$. We thus have 
 \begin{equation*}
 \begin{split}
     \alpha_n(x_1,0) &= \sum_{i=2}^{n-3} \det(w_i,w_{i+1}) \left(\prod_{2\leq j\leq n-2, j\neq i,i+1} l_j(x_1,0)\right) x_1^2 \\
     &+ (a_{n-2}-b_{n-2})\left(\prod_{2\leq j \leq n-3} l_j(x_1,0)\right) x_1\\ & + (a_2+b_2)\left(\prod_{3\leq j \leq n-2} l_j(x_1,0)\right) x_1\\
     &- \left(\prod_{2\leq j \leq n-2} l_j(x_1,0)\right) x_1 + \left(\prod_{2\leq j \leq n-2} l_j(x_1,0)\right) x_1 \\
    &= x_1^2 \cdot(\dots)\\ & + x_1 \left(\prod_{3\leq j\leq n-3} l_j(x_1,0)\right)((a_{n-2}-b_{n-2})l_2(x_1,0) + (a_2+b_2)l_{n-2}(x_1,0)).
 \end{split}
 \end{equation*}
 Similarly, we have
 \begin{equation*}
 \begin{split}
     \alpha_{n-2}(x_1,0) &= -\gamma \sum_{i=2}^{n-3} \det(w_i,w_{i+1}) \prod_{2\leq j\leq n-2, j\neq i,i+1} l_j(x_1,0) \\
     & + a_{n-2} \prod_{2\leq j\leq n-3} l_j(x_1,0) - a_2 \prod_{3\leq j\leq n-2} l_j(x_1,0) \\
     &= -\gamma(\dots) + \left(\prod_{3\leq j\leq n-3} l_j(x_1,0)\right)(a_{n-2}l_2(x_1,0) - a_2l_{n-2}(x_1,0)).
 \end{split}
 \end{equation*}
 It now suffices to prove that
 \begin{equation*}
     -\gamma((a_{n-2}-b_{n-2})l_2(x_1,0) + (a_2+b_2)l_{n-2}(x_1,0)) = x_1(a_{n-2}l_2(x_1,0) - a_2l_{n-2}(x_1,0)).
 \end{equation*}
 Using $-\gamma a_{n-2} + \gamma b_{n-2} + c_{n-2} = 0$ and $\gamma a_2 + \gamma b_2 + c_2 = 0$ we may rewrite all occurrences of $c_{n-2}$ and $c_2$ (which are implicit in $l_{n-2}$ and $l_2$) in terms of $a_{n-2},a_2,b_{n-2},b_2$, which yields
 \begin{equation*}
 \begin{split}
     &(a_{n-2} - b_{n-2})l_2(x_1,0) + (a_2 + b_2)l_{n-2}(x_1,0)  \\ =&
     (a_{n-2} - b_{n-2})(a_2x_1-\gamma(a_2+b_2)) + (a_2 + b_2)(a_{n-2}x_1 + \gamma(a_{n-2}-b_{n-2}))  \\ =&
     (2a_2a_{n-2} - a_2b_{n-2} + a_{n-2}b_2)x_1
 \end{split}
 \end{equation*}
 and
 \begin{equation*}
 \begin{split}
     &a_{n-2}l_2(x_1,0) - a_2l_{n-2}(x_1,0)  \\ =
     &a_{n-2}(a_2x_1 - \gamma(a_2+b_2)) - a_2(a_{n-2}x_1 + \gamma(a_{n-2} - b_{n-2}))  \\ =&
     - \gamma(2a_2a_{n-2} - a_2b_{n-2} + a_{n-2}b_2).
 \end{split}
 \end{equation*}
 Thus the first claim follows.

 Now let $A_n$ be smooth. For proving that  $A_{n-2} \cup A_Q$ is a contact curve to $A_n$, by B\'{e}zout's theorem it suffices to count points of contact with $A_n$. From \Cref{lem:AdjTangent} we know that every subquadrilateral of $P_{n-2}$ that shares two non-adjacent edges with both $P_n$ and $P_{n-2}$ gives one such point of contact. Every such quadrilateral is uniquely determined by two such edges of $P_{n-2}$. This gives rise to exactly
 \begin{equation*}
     \binom{n-3}{2} - \underbrace{(n-4)}_\textrm{unordered pairs of adjacent edges}
 \end{equation*}
 many such quadrilaterals and thus points of contact.
 Furthermore, the line $A_Q$ is tangent to $A_n$ at one point, and intersects $A_n$ in $n-5$ further points that all lie on $A_{n-2}$ (as shown in the first part of the proof). In total this gives rise to
 \begin{equation*}
     \binom{n-3}{2} - (n-4) + (n-5) + 1 = \frac{(n-3)(n-4)}{2}
 \end{equation*}
 contact points and we conclude by B\'{e}zout's theorem that $A_{n-2} \cup A_Q$ is a contact curve to $A_n$. Its contact divisor can be determined from the same considerations. 
 Namely, we have seen above that
 \begin{align*}
     A_Q.A_n&=2T_{1,n-1}+(R_1+\cdots+R_{n-5}),\textrm{ and}\\
     A_{n-2}.A_n&=2\sum_{i=2}^{n-4}\sum_{j=i+2}^{n-2} T_{ij}+(R_1+\cdots+R_{n-5}),
 \end{align*}
 where $T_{ij}$ is the intersection point of $L_i$ and $L_j$ and $T_{1,n-1},R_1,\ldots,R_{n-5}$ are the intersection points of $A_Q$ and $A_n$. Furthermore, we have
 \begin{align*}
     L_{n-1}.A_n&=T_{1,n-1}+\sum_{i=2}^{n-3}T_{i,n-1},\textrm{ and}\\
     A_{n-1}.A_n&=2\sum_{i=2}^{n-3}\sum_{j=i+2}^{n-1}T_{ij},
 \end{align*}
 where the second equality follows from \Cref{thm:adjointcontact}. From this we obtain
 \begin{equation*}
     A_{n-1}.A_n+A_Q.A_n=A_{n-2}.A_n+2L_{n-1}.A_n
 \end{equation*}
 which is equivalent to the claim.
\end{proof}

The following proposition is almost \Cref{thm:recursiveDetRep}.

\begin{prop} \label{prop:almostdetrep}
 Let $P$ be a convex polygon with $n$ vertices $v_1, \dots, v_n$, $n\geq4$, whose adjoint curve is smooth. There exists a symmetric determinantal representation $M$ of $\alpha_P$ satisfying the following:
 \begin{enumerate}
     \item $M$ is \emph{tridiagonal}, i.e., if $M_{ij} \neq 0$, then $|i-j| \leq 1$.
     \item The $k$-th leading submatrix of $M$ is a determinantal representation of the adjoint of the convex hull of $v_1, \dots, v_{k+3}$.
 \end{enumerate}
\end{prop}

\begin{proof}
 We adopt the notation of \Cref{prop:quadrilrest} and its proof for $P_k, A_k, \alpha_k, Q, A_Q$ and define $\alpha_Q$ in analogy. 
 We will prove the claim by induction on $n$. For $n=4$ the claim is immediate. Thus assume $n\geq5$ and let $M_{n-1}$ be a determinantal representation of $\alpha_{n-1}$ satisfying $(i)$ and $(ii)$. In particular, its leading submatrix $M_{n-2}$ of size $n-5$ is a determinantal representation of $\alpha_{n-2}$, so that $\det(M_{n-1}) = \gamma_{n-1}\alpha_{n-1}$ and $\det(M_{n-2}) = \gamma_{n-2}\alpha_{n-2}$ for some non-zero $\gamma_{n-1},\gamma_{n-2}\in\R$. 
 By \Cref{prop:quadrilrest} we know that the polynomials $\alpha_{n-1}\alpha_{n-2}\alpha_Q$ and $\alpha_{n-2}^2l_{n-1}^2$ cut out the same points with the same multiplicities on $A_n$, and hence so do $\alpha_{n-2}l_{n-1}^2$ and $\alpha_{n-1}\alpha_Q$. By Max Noether's fundamental theorem \cite{Fulton2008Curves}*{\S5.5} there exists an identity
 \begin{equation*}
     \lambda\alpha_Q\alpha_{n-1} - \mu l_{n-1}^2\alpha_{n-2} = \alpha_n
 \end{equation*}
 for some non-zero scalars $\lambda,\mu\in\R$.  Then we write
 \begin{equation*}
     M = \begin{pmatrix}
         M_{n-1} & (0, \dots, l_{n-1})^t \\
         (0, \dots, 0, l_{n-1}) & \frac{\lambda\gamma_{n-2}}{\mu\gamma_{n-1}}\alpha_Q
     \end{pmatrix}
 \end{equation*}
 and observe that
 \begin{equation*}
 \begin{split}
     \frac{\mu}{\gamma_{n-2}}\det(M) & = \frac{\lambda}{\gamma_{n-1}}\alpha_Q\det(M_{n-1}) - \frac{\mu}{\gamma_{n-2}}l_{n-1}^2\det(M_{n-2})  \\
     &= \lambda\alpha_Q\alpha_{n-1} - \mu l_{n-1}^2\alpha_{n-2} = \alpha_n.\qedhere
 \end{split}
 \end{equation*}
\end{proof}

It remains to prove that the determinantal representation from \Cref{prop:almostdetrep} is definite and extend it to the possibly singular case.
In order to apply \Cref{thm:interlmakesdef} to certify definiteness, we need to show that $\alpha_P$ is hyperbolic and that $\alpha_{P'}$ is an interlacer. Hyperbolicity of $\alpha_P$ is proven in \cite{KohnEtAl2024AdjCurves}*{Theorem 3.8}. Our proof that $\alpha_{P'}$ is an interlacer of $\alpha_P$ is also based on this theorem. 

\begin{lem}\label{cor:Interlacer}
    The polynomials $\alpha_P$ and $\alpha_{P'}$ as in \Cref{thm:adjointcontact} are strictly hyperbolic with respect to every point in the interior of the polygon $P$. Moreover, the polynomial $\alpha_{P'}$ interlaces $\alpha_P$ with respect to all such points.
\end{lem}

\begin{proof}   
    By \cite{KohnEtAl2024AdjCurves}*{Theorem 3.8} the defining polynomials $\alpha_P$ and $\alpha_{P'}$ of the curves $A$ and $A'$ are strictly hyperbolic in the sense of \Cref{prop:hyperbRealLocus}.

    Let $(O_i)_i$ be the ovals of $A(\R)$ (labeled as in \Cref{prop:hyperbRealLocus}). Furthermore, let $K_0, \dots, K_N$ be the connected components of $\R^2\backslash A(\R)$ where $\overline{K_i} \cap O_j \neq \emptyset$ if and only if $i \leq j \leq i+1$. Next let $(O_i')_i$ be the ovals of $A'(\R)$. Since $A$ and $A'$ are contact curves, for all $i$ there exists $j$ such that $O_i' \subseteq \overline{K_j}$. Following \cite{KohnEtAl2024AdjCurves}*{Theorem 3.8} $O_i'$ intersects both $O_{i}$ and $O_{i+1}$, such that $O_i' \subseteq \overline{K_i}$. To conclude the proof we have to show that $O_i'$ is not contractible in $\overline{K_i}$, i.e., that its interior contains $O_i$. By hyperbolicity of $A'$ it suffices to do so for $i = 1$. If $O_1'$ were contractible in $K_1$, then the polygon $P$ would be contained in exterior of $O_1'$, contradicting that $A'$ is hyperbolic with respect to the interior of a subpolygon of $P$.    
\end{proof}

\begin{proof}[Proof of \Cref{thm:recursiveDetRep}]
    If the adjoint of $P$ is smooth, then the claim follows directly from \Cref{prop:almostdetrep}, \Cref{thm:interlmakesdef} and \Cref{cor:Interlacer}. Now let $P$ be a polygon with singular adjoint curve $A = \cV(\alpha_P)$. Since the adjoint continuously depends on its residual arrangement (\cite{KohnRanestad2019AdjCurves}*{Corollary 1}) and the generic adjoint is smooth (\cite{telen2025toric}*{Corollary 7.14}), we may thus choose a sequence of polygons $(P_k)_{k=0}^\infty$ converging to $P$ with corresponding adjoint polynomials $\alpha_k$ such that $\cV(\alpha_k)$ is smooth for all $k$ and $\alpha_P = \lim_{k\to\infty} \alpha_k$. Furthermore, we may assume that there exists a point $e$ in the interior of all $P_k$ and $P$. Then each $\alpha_k$ is hyperbolic with respect to $e$. Now it follows from \cite{PlaumannVinzantHyp2013LDR}*{Lemma 3.4} that a suitable limit of determininantal representations of the $\alpha_k$ is a determinantal representation with the desired properties as these properties are all stable under degeneration.
\end{proof}

\begin{rem} \label{rem:detRepExpl}
 The entries of the determinantal representation from \Cref{thm:recursiveDetRep} can be described very explicitly. To do so let $M = (M_{ij})_{i,j}$. Then it follows from the proof of \Cref{prop:almostdetrep} that we have up to a (possibly zero) scalar that
 \begin{equation*}
     M_{ij} = \left\lbrace \begin{array}{ll}
         0 & |i-j| \geq 2, \\
         l_{\min\{i,j\}+3} & |i-j| = 1,\\
         \alpha_{Q_i} & i = j,
     \end{array}\right.
 \end{equation*}
 where $Q_i$ is the subquadrilateral of $P$ with vertices $v_1, v_{i+1}, v_{i+2}, v_{i+3}$ and $\alpha_{Q_i}$ is its adjoint polynomial. Indeed, the proof of \Cref{prop:almostdetrep} shows this  in the smooth case and these properties are stable under degeneration.
\end{rem}

We conclude this section with a computational example.

\begin{ex} \label{ex:planar}
    Consider the heptagon $P$, whose vertex set is
    \begin{equation*}
        V(P) = \{(0,-2),(3,0),(3,2),(0,3),(-2,2),(-4,0),(-2,-2)\}.
    \end{equation*}
    For a visualization of $P$ and its adjoint we refer to \Cref{fig:WLOGsubquadril}. Its defining facet equations are the following seven linear forms
    \begin{small}
    \begin{align*}
        l_1(x_1,x_2) &= -x_1+x_2,& l_2(x_1, x_2) &= -2x_1-x_2+9, & l_3(x_1, x_2) &= -x_1+3x_2+20,\\
        l_4(x_1, x_2) &= x_1-x_2+8,& l_5(x_1, x_2) &= x_1+3,&l_6(x_1, x_2) &= x_1+x_2,\\
        l_7(x_1,x_2) &= -x_2+1,&&&&
    \end{align*}
    \end{small}its adjoint polynomial is
    \begin{tiny}
    \begin{equation*}
    \alpha_P = 2x_1^4+2x_1^3x_2-7x_1^2x_2^2+6x_1x_2^3-3x_2^4+44x_1^3-26x_1^2x_2+4x_1x_2^2+74x_2^3+199x_1^2-824x_1x_2-47x_2^2-2880x_2
    \end{equation*}
    \end{tiny}and it has a smooth zero set.
    Following the proof of \Cref{prop:almostdetrep} one obtains the following determinantal representation of $\alpha_P$:
    \begin{equation*}
        \begin{pmatrix}
            27x_1-39x_2+316 & x_1 - x_2 + 8 & 0 & 0 \\
            x_1 - x_2 + 8 & \frac{3}{49}x_1 - \frac{5}{147}x_2+\frac{20}{49} & x_1+3 & 0 \\
            0 & x_1+3 & \frac{147}{4}x_1 + \frac{147}{8}x_2 + \frac{441}{4} & x_1 + x_2 \\
            0 & 0 & x_1 + x_2 & \frac{32}{147}x_2
        \end{pmatrix}.
    \end{equation*}
The structure of this determinantal representation as described in \Cref{thm:recursiveDetRep} and \Cref{rem:detRepExpl} is clearly visible.
\end{ex}

\section{Three Dimensions} \label{sec:4}\label{sec:detreps3d}
Next we study determinantal representations of adjoints of three-dimensional polytopes. Our main auxiliary notion here is that of a {nice} line arrangement.

\begin{Def}\label{def:nicearrangements}
    Let $D\geq1$. An arrangement $C$ of lines in $\pp^3$ is \emph{nice for degree $D+1$} if $C$ is the union of a line arrangement $Y\subseteq\pp^3$ which is nice for degree $D$ and an arrangement $Z\subseteq\pp^3$ of $D$ disjoint lines such that
    \begin{enumerate}
        \item no line of $Y$ is contained in $Z$,
        \item every line of $Z$ intersects $D-1$ lines from $Y$,
        \item there exists a plane $H\subseteq\pp^3$ such that the closure of $X\smallsetminus H$ is nice for degree $D$,
        \item no three lines of $C$ intersect in a point.
    \end{enumerate}
    We further declare the empty arrangement to be the only one which is nice for degree one.
\end{Def}

\begin{rem}\label{niceremarks}
    It follows from the definition that an arrangement $C$ that is nice for degree $D+1$ consists of exactly $\binom{D+1}{2}$ lines. This also implies that the hyperplane $H$ from part (iii) in \Cref{def:nicearrangements} contains $D$ lines of $C$.
\end{rem}

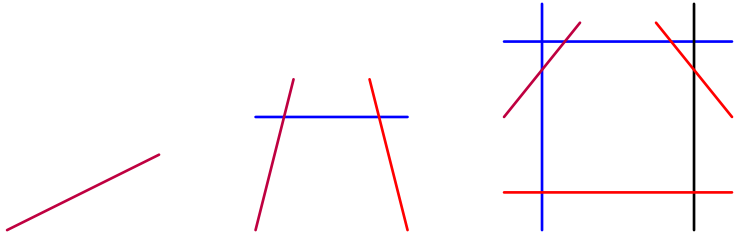
\begin{figure}
\begin{tikzpicture}[line cap=round,line join=round,x=1cm,y=1cm]

\draw [line width=1pt,color=purple] (0,0)-- (2,1);

\end{tikzpicture}
\hspace{1cm}
\begin{tikzpicture}[line cap=round,line join=round,x=1cm,y=1cm]

\draw [line width=1pt,color=blue] (0,1.5)-- (2,1.5);

\draw [line width=1pt,color=purple] (0,0)-- (0.5,2);

\draw [line width=1pt,color=red] (1.5,2)-- (2,0);

\end{tikzpicture}
\hspace{1cm}
\begin{tikzpicture}[line cap=round,line join=round,x=1cm,y=1cm]

\draw [line width=1pt,color=blue] (-0.5,2)-- (2.5,2);

\draw [line width=1pt,color=blue] (0,2.5)-- (0,-0.5);

\draw [line width=1pt] (2,2.5)-- (2,-0.5);

\draw [line width=1pt,color=red] (-0.5,0)-- (2.5,0);

\draw [line width=1pt,color=purple] (0.5,2.25)-- (-0.5,1);

\draw [line width=1pt,color=red] (1.5,2.25)-- (2.5,1);

\end{tikzpicture}
    \caption{Line arrangements that are nice for degree $D+1$ for $D=1,2,3$ (from left to right).}
    \label{fig:arrangements}
\end{figure}  

\begin{ex}
    \Cref{fig:arrangements} shows line arrangements that are nice for degree $D+1$ for $D=1,2,3$. In each case, there are $D$ disjoint lines that are colored red or purple that intersect the remaining lines in $D-1$ points. Removing them gives an arrangement that is nice for degree $D$. Furthermore, there are $D$ lines that are colored blue or purple. These are contained in a hyperplane and removing them gives an arrangement that is nice for degree $D$.
\end{ex}

We will prove that the ideal sheaf of a nice arrangement satisfies the conditions from \Cref{thm:idealsheafdetrep}. We note that in general it might be hard to compute the cohomology groups of the ideal sheaf of line arrangements just from their combinatorics: it is already highly non-trivial for the case of $m$ pairwise disjoint lines \cite{HartshorneHirschowitz}. However, the definition of a nice arrangement was chosen in a way that makes the computation of the relevant groups feasible.

\begin{lem}\label{lem:smallestform}
    If $C\subseteq\pp^3$ is a line arrangement that is nice for degree $D$, then $h^0(\pp^3,\cI_C(D-2))=0$ and $h^0(\pp^3,\cI_C(D-1))\neq0$. In other words, the smallest degree of a non-zero form that vanishes on $C$ is equal to $D-1$.
\end{lem}

\begin{proof}
    We prove the claim by induction on $D$. The claim is trivial for $D=1$. Assume that the claim is true for all arrangements that are nice for degree $D$ and let $C$ be nice for degree $D+1$. If $H$ is a hyperplane as in condition (iii) of \Cref{def:nicearrangements}, defined by a linear polynomial $L$, then it contains $D$ lines by \Cref{niceremarks}. Therefore, every form $G$ of degree $D-1$ that vanishes on $C$ must vanish on $H$. Thus $G=L\cdot G'$ for some form $G'$ of degree $D-2$ that vanishes on the closure of $C\smallsetminus H$. Since this is an arrangement that is nice for degree $D$, the induction hypothesis implies that $G'=0$ and thus $G=0$. This shows that no non-zero form of degree $D-1$ vanishes on an arrangement that is nice for degree $D+1$. To prove the other claim, choose a non-zero form $F$ of degree $D-1$ that vanishes on $C\smallsetminus H$. This exists by induction hypothesis. Then $L\cdot F$ is non-zero, vanishes on $C$ and has the correct degree $D$.
\end{proof}

\begin{thm}\label{thm:nicearrangementsdetrep}
    The ideal sheaf $\cI_C$ of  a line arrangement $C\subseteq\pp^3$ that is nice for degree $D$ satisfies the conditions from \Cref{thm:idealsheafdetrep}:
    \begin{align*}
        h^0(\pp^3,\cI_C(D-2))&=h^1(\pp^3,\cI_C(D-2))\\
        &=h^1(\pp^3,\cI_C(D-3))=0,\\
        h^2(\pp^3,\cI_C(D-3))&=h^3(\pp^3,\cI_C(D-3)).
    \end{align*}
\end{thm}

\begin{proof}
    We prove the claim by induction on $D$. For $D=1$ we have $\cI_C=\cO_{\pp^3}$ and the claim is true. Assume that the claim is true for all arrangements that are nice for degree $D$ and let $C$ be nice for degree $D+1$. By \Cref{def:nicearrangements} we can write $C=Y\cup Z$ where $Y$ is a line arrangement that is nice for degree $D$ and $Z$ consists of $D$ disjoint lines, each of which intersects $D-1$ lines from $Y$. Let $\cI_Y$ be the ideal sheaf of $Y$ on $\pp^3$. The kernel of the map $\cI_Y\to\cO_Z$ is $\cI_C$ because $C=Y\cup Z$. Because $Z$ is the disjoint union of $D$ lines, we have $\cO_Z\simeq\cO_{\pp^1}^D$. Since the intersection of each of these lines with $Y$ is transversal and consists of $D-1$ points, the image of $\cI_Y\to\cO_Z$ can be identified with $\cO_{\pp^1}(-(D-1))^D$.    
    This gives the short exact sequence
    \begin{equation}\label{eq:sequencenicearrangement}
        0\to\cI_C\to\cI_Y\to\cO_{\pp^1}(-(D-1))^D\to0.
    \end{equation}
    Next, we note that by \Cref{lem:smallestform} there is a hypersurface of degree $D$ that contains~$Y$. Thus, by  the induction hypothesis the ideal sheaf $\cI_Y$ satisfies the conditions from \Cref{thm:idealsheafdetrep} and we get
    \begin{align}
        h^1(\pp^3,\cI_Y(k+D-1))&=0&\textrm{ for }k\in\Z,\label{eq:coho1} \\
        h^2(\pp^3,\cI_Y(k+D-1))&=h^3(\pp^3,\cI_Y(k+D-1))&\textrm{ for }k\geq-2,\label{eq:coho2}\\
        h^0(\pp^3,\cI_Y(D-1))&=D.\label{eq:cohoD}
    \end{align}
    \Cref{eq:coho1} implies $h^1(\cI_Y(D-1))=0$. Thus the first part of the long exact sequence in cohomology obtained from twisting \Cref{eq:sequencenicearrangement} by $D-1$ is
    \begin{align*}
        0&\to H^0(\pp^3,\cI_C(D-1))\to H^0(\pp^3,\cI_Y(D-1))\to H^0(\pp^3,\cO_{\pp^1}^D)\\&\to H^1(\pp^3,\cI_C(D-1))\to0.
    \end{align*}
    We have $h^0(\cO_{\pp^1}^D)=D$ and by \Cref{eq:cohoD} we also have $h^0(\cI_Y(D-1))=D$. This shows that $h^0(\cI_C(D-1))=h^1(\cI_C(D-1))$. \Cref{lem:smallestform} shows that both are zero. Next, we look at the long exact sequence in cohomology obtained from twisting \Cref{eq:sequencenicearrangement} by $D-2$ and obtain:
    \begin{equation*}
        0\to H^i(\pp^3,\cI_C(D-2))\to H^i(\pp^3,\cI_Y(D-2))\to 0
    \end{equation*}
    for $i=1,2,3$ where we used that $h^j(\cO_{\pp^1}(-1)))=0$ for all $j$. Now the remaining claims follow from \Cref{eq:coho1} and \Cref{eq:coho2} for $k=-1$.
\end{proof}

\begin{cor}\label{cor:smoothdetrep}
    Let $P$ be a three-dimensional polytope with at most eight facets such that $\cH_P$ is simple. Then its adjoint has a determinantal representation. In particular, any smooth adjoint of a three-dimensional polytope with simple facet hyperplane arrangement has one. 
\end{cor}

\begin{proof}
    By inspection of \cite{BrittonDunitz73Polytopes}*{Fig. 4 and Fig. 5}, see also \Cref{tab:combinatorialTypes}, we find in the residual arrangement of each combinatorial type a line arrangement as in \Cref{fig:arrangements} that is nice for the respective degree. Thus the claim follows from \Cref{thm:nicearrangementsdetrep}.
    The additional claim on smooth adjoints then follows from \Cref{thm:smooth3d}.
\end{proof}

\begin{rem}
    Let $P$ be a three-dimensional polytope with $D+4$ facets. If the residual arrangement $\cR(P)$ of $P$ contains an arrangement that is nice for degree $D$, then it contains in particular $D-1$ disjoint lines. Since every such line is the intersection of two facet hyperplanes and because lines that lie in the same hyperplane intersect, we have
    \begin{equation*}
        2(D-1)\leq D+4,
    \end{equation*}
    which implies that $D\leq6$. Therefore, a polytope whose residual arrangement contains a line arrangement that is nice for the degree of its adjoint can have at most ten facets. This does not rule out the existence of a possibly different line arrangement $C$ in $\cR(P)$ that satisfies the conditions from \Cref{thm:idealsheafdetrep}, see \Cref{ex:non-nicearrangement}. Moreover, the adjoint hypersurface often contains more lines than just the ones from the residual arrangement. Namely, it contains the adjoint curve of every facet. For quadrilateral facets this is a line.
\end{rem}

Note that \Cref{thm:nicearrangementsdetrep} is a general statement about line arrangements in $\mathbb{P}^3$ and that any surface containing a nice line arrangement of correct degree has a determinantal representation. In particular, one can also use \Cref{thm:nicearrangementsdetrep} to look for determinantal representations of adjoints of arbitrary three-dimensional polytopes, whose facet hyperplane arrangement is not necessarily simple. To demonstrate this, in \Cref{ex:exmplnonsimple} we present a non-simple polytope whose residual arrangement contains a nice line arrangement. The reason we restrict ourselves to the case of simple polytopes in \Cref{cor:smoothdetrep} is that there are too many combinatorial types of non-simple polytopes to perform an exhaustive search: \cite{BrittonDunitz73Polytopes}*{Fig. 5} contains $257$ combinatorial types of polytopes with eight facets, only $14$ of which lead to a simple facet hyperplane arrangement. For the same reason we did not go through all combinatorial types of simple polytopes with nine or ten facets.

\begin{ex}\label{ex:exmplnonsimple}
Consider a polytope $P$ with combinatorial type as in \Cref{fig:exmplnonsimple}. Then $P$ is a non-simple polytope with seven facets. Its adjoint surface has degree three.
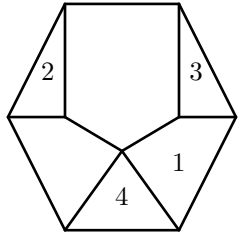
\begin{figure}[h]
    \centering
    \begin{tikzpicture}[line cap=round,line join=round,x=1.5cm,y=1.5cm]
\draw [line width=1pt] (1,0)-- (0.5,1);
\draw [line width=1pt] (0.5,1)-- (-0.5,1);
\draw [line width=1pt] (-0.5,1)-- (-1,0);
\draw [line width=1pt] (-1,0)-- (-0.5,-1);
\draw [line width=1pt] (-0.5,-1)-- (0.5,-1);
\draw [line width=1pt] (0.5,-1)-- (1,0);

\draw [line width=1pt] (0.5,1)-- (0.5,0);
\draw [line width=1pt] (0.5,0)-- (0,-0.3);
\draw [line width=1pt] (-0.5,1)-- (-0.5,0);
\draw [line width=1pt] (-0.5,0)-- (0,-0.3);
\draw [line width=1pt] (0.5,0)-- (1,0);
\draw [line width=1pt] (-0.5,0)-- (-1,0);
\draw [line width=1pt] (0,-0.3)-- (-0.5,-1);
\draw [line width=1pt] (0,-0.3)-- (0.5,-1);

\node at (0.5,-0.4) {1};
\node at (-0.65,0.4) {2};
\node at (0.65,0.4) {3};
\node at (0,-0.7) {4};
\end{tikzpicture}\caption{Nonsimple polytope with a nice line arrangement.}
    \label{fig:exmplnonsimple}
\end{figure}
Its residual arrangement contains five lines, among them the three lines $R_{12}, R_{23}$ and $R_{34}$ (where $R_{ij}$ denotes the line obtained by intersecting the hyperplanes corresponding to the facets with labels $i$ and $j$). These three lines form a nice arrangement for degree three, which in turn yields a determinantal representation of the adjoint of $P$. Note that since the facets with labels 2, 3 and 4 are pairwise non-adjacent, the adjoint of $P$ is singular by \Cref{prop:adj_sing}.
\end{ex}

\begin{ex}\label{ex:non-nicearrangement}
    By direct computation one can show that the ideal sheaf of three lines in $\pp^3$ that intersect in one common point satisfies the conditions of \Cref{thm:idealsheafdetrep} for $D=3$ if and only if they do not lie in a common plane. This shows in particular that there are line arrangements that are not nice in the sense of \Cref{def:nicearrangements} but still give rise to a determinantal representation. Such arrangements can be contained in an adjoint of degree $D=3$: take the lines $R_{23},R_{24},R_{34}$ from the previous example \Cref{ex:exmplnonsimple}.
\end{ex}

Let $P \subseteq \pp^3$ be a polytope whose adjoint hypersurface $X$ contains a nice arrangement as defined in \Cref{def:nicearrangements}. 
Computing a determinantal representation for~$X$ explicitly is not difficult using the computer algebra system \texttt{Macaulay2} \cite{M2}.
The following code returns such a determinantal representation.

\begin{verbatim}
S = QQ[x0,x1,x2,x3];

-----
-- !assume the following as given!
-- adjoint polynomial a defining the adjoint surface V(a)
-- nice arrangement given by its vanishing ideal D
-----

-- computation of a determinantal representation
Ma = D*(S^1/(ideal(a)*S^1));
rso = (res Ma).dd;
DetRep = rso_1;

print(ideal(det(DetRep)) == ideal(a));
-- returns true
\end{verbatim}

\begin{ex}
 Consider the polytope $P \subseteq \R^3$ determined by the intersection of the halfspaces $f_i(1,x_1,x_2,x_3)\geq 0$ where the $f_i$ are defined below:
 \begin{align*}
     f_0 &= x_0 + x_1 - 3x_2 + 2x_3, &
     f_1 &= x_0 + x_1 + 3x_2 + \frac{1}{5}x_3, \\
     f_2 &= x_0 + 3x_1 - x_2, &
     f_3 &= x_0 - x_1 - 3x_2 + \frac{3}{2}x_3, \\
     f_4 &= x_0 - 3x_1 - x_2 - \frac{2}{3}x_3, &
     f_5 &= x_0 - x_1 + 3x_2 - \frac{1}{2}x_3, \\
     f_6 &= x_0 + 3x_1 + x_2 - x_3, &
     f_7 &= x_0 - 3x_1 + x_2 - \frac{3}{2}x_3.
 \end{align*}
 The combinatorial type of $P$ corresponds to the ninth entry of \Cref{tab:combinatorialTypes}. Writing $R_{ij} = \cV(f_i,f_j)\subseteq\pp^3$, in this case a nice arrangement as in \Cref{fig:arrangements}, that is contained in the adjoint of $P$, is given by the union of the following residual lines:
 \begin{equation*}
    R_{04}, R_{05}, R_{07}, R_{17}, R_{27}, R_{14}. 
 \end{equation*}
 The resulting determinantal representation of the adjoint $\alpha_P$ of $P$ is
 \begin{tiny}
 \begin{equation*}
     \begin{pmatrix}
         x_0+\frac{3}{4}x_2+\frac{59}{160}x_3  &  \frac{3}{4}x_2-\frac{91}{160}x_3  &  -\frac{9}{4}x_2+\frac{3}{160}x_3  &  -\frac{3003}{1129}x_1-\frac{38685}{4516}x_2+\frac{473893}{180640}x_3 \\[3pt]
         -x_1+\frac{11}{6}x_2-\frac{283}{240}x_3  &  x_0-\frac{3}{2}x_2+\frac{137}{240}x_3  &  \frac{3}{2}x_2-\frac{137}{80}x_3  &  \frac{3836}{1129}x_1+\frac{13337}{2258}x_2-\frac{1452191}{270960}x_3 \\[3pt]
         -\frac{19}{12}x_2+\frac{37}{96}x_3  &  -x_1+\frac{3}{4}x_2-\frac{77}{96}x_3  &  x_0+\frac{3}{4}x_2-\frac{1}{32}x_3  &  -\frac{5902}{1129}x_1-\frac{10645}{4516}x_2+\frac{404243}{108384}x_3 \\[3pt]
         \frac{1}{3}x_2-\frac{1}{15}x_3  &  \frac{7}{120}x_3  &  -x_1-\frac{3}{10}x_3  &  x_0-\frac{668}{1129}x_1-\frac{500}{1129}x_2-\frac{112141}{135480}x_3
     \end{pmatrix}.
 \end{equation*}
 \end{tiny}
\end{ex}

\section{Universal Adjoints and the ABHY Associahedron} \label{sec:6}
In this section we describe another hypersurface associated to $P$, called a \emph{universal adjoint}. The term was coined in \cite{telen2025toric} in a more general setting of polyhedral fans. We now give a definition following this reference. Let $\Sigma_P$ denote the normal fan of $P$ in an appropriate affine chart of $\mathbb{P}^n$, and let 
\begin{equation*}
P = \{y\in \mathbb{R}^n\mid \langle u_\rho , y\rangle + z_\rho \geq 0\ \text{ for } \rho\in\Sigma_P(1)\}    
\end{equation*}
 be a minimal facet description of $P$ in the same chart. Here $\Sigma_P(1)$ denotes the set of rays of $\Sigma_P$ and $\langle \cdot,\cdot\rangle$ is the Euclidean inner product. We record the vectors $u_\rho$ in a $d \times n$ matrix $U$. The \emph{universal adjoint polynomial} $\mathrm{Adj}_P$ is then a polynomial in the variables $x_\rho$ for $\rho\in \Sigma_P(1)$ defined as follows:
$$\mathrm{Adj}_P(x) = \sum_{\sigma \in \Sigma_P(n)} |\det U_\sigma|\cdot\prod_{\rho\not\in\sigma(1)} x_\rho.$$
Here $\Sigma_P(n)$ is the set of $n$-dimensional cones of $\Sigma_P$ and $U_\sigma$ is the submatrix of $U$ consisting of the columns indexed by the rays of $\sigma$.
When substituting $\langle u_\rho , y\rangle + z_\rho$ for $x_\rho$ in $\mathrm{Adj}_P(x)$, we get the restriction of the adjoint $\alpha_P$ to the considered affine chart, see e.g. \cite{PavlovTelen2024PosGeom}*{Section 2}. The adjective ``universal'' underscores the fact that $\mathrm{Adj}_P$ only depends on the normal fan of $P$, and that the adjoints of all polytopes with the same normal fan as $P$ can be obtained from $\mathrm{Adj}_P$ as described above.

\begin{rem}
Note that while the adjoint $\alpha_P$ of an $n$-dimensional polytope $P$ with $k$ facets has degree $k-n-1$, the degree of the universal adjoint is $k-n$ \cite{PavlovTelen2024PosGeom}*{Proposition 2.5}. For this reason, substituting $\langle u_\rho , y\rangle + z_\rho$ for $x_\rho$ in a determinantal representation $M(x)$ of $\mathrm{Adj}_P$ gives a matrix with linear entries whose determinant is $\alpha_P$. However, as its size is larger than the degree of $\alpha_P$, it is not a determinantal representation in the sense of \Cref{def:detrep}, cf. \Cref{rem:detrep}.
\end{rem}

While we do not pursue the question of whether universal adjoints admit a determinantal representation in general, we study it in detail for the ABHY associahedron \cite{arkani2018scattering}. This example is of direct relevance to physics, and served as the initial motivation for this article. 

The ABHY associahedron is just a particular metric realization of the associahedron; for details, see \cite{arkani2018scattering}*{Section 3}. The vertices of the $(n-3)$-dimensional associahedron $\mathcal{A}_{n-3}$ correspond to triangulations of a regular planar $n$-gon, and the facets of $\mathcal{A}_{n-3}$ are labeled by individual diagonals of the $n$-gon. In the universal adjoint, each facet corresponds to a variable. We will therefore denote these variables $X_{ij}$, where $(i,j)$ is the diagonal connecting vertices $i$ and $j$ in the $n$-gon (we enumerate the vertices  counterclockwise), following the physics convention. 

Let $\mathcal{T}_{n}$ be the set of all triangulations of a regular $n$-gon. Then the polynomial we consider is given by 
$$\mathrm{Adj}_{n-3}=\sum_{T\in\mathcal{T}_{n}}\prod_{(i,j)\not\in T} X_{ij}.$$
This is the numerator of the rational function giving the tree-level planar bi-adjoint scalar $n$-point amplitude in $\phi^3$ quantum field theory. In \cite{arkani2024hidden}*{page 46} the authors speculate whether the polynomial $\mathrm{Adj}_{n-3}$ is ``expressible as the determinant of a predictable matrix, in
a way that would make all our hidden zeros manifest''.
We answer this question to the negative for the following natural type of determinantal representations for multi-affine polynomials.
\begin{Def}[AV-representations]\label{def:avrepresentation}
   Let $f\in\C[x_1,\ldots,x_n]$ be a polynomial that is multi-affine in the variables $x_1,\ldots,x_d$, i.e., each appears to degree at most one, such that the coefficient of the monomial $x_1\cdots x_d$ is one. We say that $f$ has an \emph{AV-representation} (with respect to $x_1,\ldots,x_d$) if there is a $d\times d$ matrix $A$ whose entries are polynomials in the variables $x_{d+1},\ldots,x_n$ such that 
   \begin{equation*}
       f=\det(\textrm{diag}(x_1,\ldots,x_d)+A).
   \end{equation*}
   We call the variables $x_1,\ldots,x_d$ the \emph{primary variables} and the $x_{d+1},\ldots,x_n$ the \emph{secondary variables} of the AV-representation.
\end{Def}

\begin{rem}\label{rem:comparison}
    AV-represenations were studied extensively by Al Ahmadieh and Vinzant in \cite{al2024determinantal} (and hence the name). We note that the existence of an AV-representation is neither implied by nor it implies the existence of a determinantal representation in the sense of \Cref{def:detrep}: While \Cref{def:avrepresentation} is more restrictive than \Cref{def:detrep} regarding the primary variables by requiring them to appear only on the diagonal, it is more flexible on the secondary variables by allowing for higher degrees. The AV-representations that we construct in this section, however, are special in that they only have linear entries, therefore being determinantal representations in the sense of \Cref{def:detrep}. Note that an AV-representation  makes it manifest that $f$ vanishes on all linear subspaces defined by
    \begin{equation*}
        x_i=x_{d+1}=\cdots=x_n=0,
    \end{equation*}
    where $i\in[d]$.
\end{rem}

\begin{rem}\label{rem:derivativeAVrep}
    Let $f$ be as in \Cref{def:avrepresentation} and $M$ be an AV-representation of $f$ with primary variables $x_1,\ldots,x_d$. 
    It is clear from the definition that the upper left $(d-1)\times(d-1)$ block of $M$ is an AV-representation of the partial derivative $\frac{\partial f}{\partial x_d}$ with primary variables $x_1,\ldots,x_{d-1}$.
\end{rem}
Now we turn back to the ABHY associahedron. As in the previous sections, it turns out that the situation is nice in two and three dimensions but not for the higher dimensional case.
\begin{ex}\label{ex:3dassoc}
    The universal adjoint of the $3$-dimensional ABHY associahedron  is the following multi-affine homogeneous polynomial of degree six:
    \begin{align*} A_3 =& X_{13}X_{14}X_{15}X_{24}X_{25}X_{35}+X_{14}X_{15}X_{24}X_{25}X_{26}X_{35}+X_{13}X_{14}X_{15}X_{25}X_{35}X_{36}+\\&X_{13}X_{15}X_{25}X_{26}X_{35}X_{36}+X_{15}X_{24}X_{25}X_{26}X_{35}X_{36}+X_{13}X_{14}X_{15}X_{24}X_{25}X_{46}+\\&X_{14}X_{15}X_{24}X_{25}X_{26}X_{46}+X_{13}X_{14}X_{15}X_{24}X_{36}X_{46}+X_{13}X_{14}X_{24}X_{26}X_{36}X_{46}+\\&X_{14}X_{24}X_{25}X_{26}X_{36}X_{46}+X_{13}X_{14}X_{15}X_{35}X_{36}X_{46}+X_{13}X_{14}X_{26}X_{35}X_{36}X_{46}+\\&X_{13}X_{25}X_{26}X_{35}X_{36}X_{46}+X_{24}X_{25}X_{26}X_{35}X_{36}X_{46}.
    \end{align*}
    This polynomial has $14$ terms, one for each vertex of the associahedron. Each term is the product of all variables labeled by the edges that do not appear in a given triangulation of the hexagon. The following matrix $M$ provides a remarkably simple AV-representation of this polynomial:
    $$M =  \begin{pmatrix}
        X_{13}& X_{25}& X_{24}& X_{15}& 0& X_{15}\\ -X_{24}& X_{14} + X_{25}& X_{24}& X_{15}& 0& X_{15}\\ -X_{25}& X_{25}& X_{24} + X_{35}& 0& X_{15}& X_{15}\\ 0& X_{25}& X_{24}& X_{15} + X_{26}& 0& X_{15}\\ X_{25}& 0& 0& X_{25}& X_{36}& 0\\ 0& -X_{25}& 0& -X_{25}& X_{24}& X_{46}
    \end{pmatrix}.$$
    The secondary variables $X_{15},X_{24},X_{25}$ correspond to the edges of a triangulation called  ``the snake'', see \Cref{fig:hexagontriangulations}. Finally, we note that the top-left $3\times3$ submatrix of $M$ is an AV-representation for the universal adjoint of the $2$-dimensional ABHY associahedron, cf. \Cref{rem:derivativeAVrep}:
    \begin{equation*}
        \mathrm{Adj}_2=X_{13} X_{14} X_{24} + X_{14} X_{24} X_{25} + X_{13} X_{14} X_{35} + X_{13} X_{25} X_{35} + X_{24} X_{25} X_{35}.
    \end{equation*}
\end{ex}

\begin{rem}
    We found the AV-representation from \Cref{ex:3dassoc} heuristically by taking iterated derivatives of $\mathrm{Adj}_3$ in all but two of the primary variables and then guessing an AV-representation for the resulting degree two polynomial. By \Cref{rem:derivativeAVrep} every principal $2\times2$ submatrix of an AV-representation for $\mathrm{Adj}_3$ is an AV-representation of such a derivative.
\end{rem}

For the rest of the section we will prove the following negative result.

\begin{thm}\label{thm:noavrep}
    For $n\geq 7$ the universal adjoint $\mathrm{Adj}_{n-3}$ of the $(n-3)$-dimensional ABHY associahedron does not admit an AV-representation.
\end{thm}

We will make use of the following lemmas.

\begin{lem}[\cite{al2024determinantal}*{Theorem 3.1}]\label{lem:abeer}
    Let $f\in\C[x_1,\ldots,x_n]$ be a polynomial that is multi-affine in the variables $x_1,\ldots,x_d$. If $f$ has an {AV-representation} with respect to $x_1,\ldots,x_d$, then for every $i\neq j\in[d]$ the \emph{Rayleigh difference}    \begin{equation*}
        \Delta_{ij}(f)=\frac{\partial f}{\partial x_i}\cdot \frac{\partial f}{\partial x_j}- f\cdot \frac{\partial^2 f}{\partial x_i\partial x_j}
    \end{equation*}
    is the product of two polynomials that are multi-affine in $x_1,\ldots,x_d$. These polynomials can be chosen to be $(d-1)\times(d-1)$-subdeterminants of the AV-representation.
\end{lem}
\begin{figure}[ht]
    \begin{tikzpicture}[line cap=round,line join=round,x=1.0cm,y=1.0cm]
\clip(-1.22,-0.2) rectangle (2.23,1.89);
\draw [color=black] (0,0)-- (1,0);
\draw [color=black] (1,0)-- (1.5,0.87);
\draw [color=black] (1.5,0.87)-- (1,1.73);
\draw [color=black] (1,1.73)-- (0,1.73);
\draw [color=black] (0,1.73)-- (-0.5,0.87);
\draw [color=black] (-0.5,0.87)-- (0,0);
\draw [line width=2pt] (0,0)-- (1,0);
\draw [line width=2pt] (1,0)-- (1.5,0.87);
\draw [line width=2pt] (1.5,0.87)-- (1,1.73);
\draw [line width=2pt] (1,1.73)-- (0,1.73);
\draw [line width=2pt] (0,1.73)-- (-0.5,0.87);
\draw [line width=2pt] (-0.5,0.87)-- (0,0);
\draw [line width=2pt] (0,0)-- (1,0);
\draw [line width=2pt,color=red] (0,0)-- (0,1.73);
\draw [line width=2pt,color=red] (0,1.73)-- (1,0);
\draw [line width=2pt,color=red] (1,0)-- (1,1.73);
\begin{scriptsize}
\fill [color=black] (0,0) circle (2pt);
\fill [color=black] (1,0) circle (2pt);
\fill [color=black] (1.5,0.87) circle (2pt);
\fill [color=black] (1,1.73) circle (2pt);
\fill [color=black] (0,1.73) circle (2pt);
\fill [color=black] (-0.5,0.87) circle (2pt);
\draw[color=black] (-0.2,-0.1) node {$1$};
\draw[color=black] (1.2,-0.1) node {$2$};
\draw[color=black] (1.7,0.87) node {$3$};
\draw[color=black] (1.2,1.74) node {$4$};
\draw[color=black] (-0.2,1.74) node {$5$};
\draw[color=black] (-0.7,0.87) node {$6$};
\end{scriptsize}
\end{tikzpicture}\begin{tikzpicture}[line cap=round,line join=round,x=1.0cm,y=1.0cm]
\clip(-1.22,-0.2) rectangle (2.23,1.89);
\draw [color=black] (0,0)-- (1,0);
\draw [color=black] (1,0)-- (1.5,0.87);
\draw [color=black] (1.5,0.87)-- (1,1.73);
\draw [color=black] (1,1.73)-- (0,1.73);
\draw [color=black] (0,1.73)-- (-0.5,0.87);
\draw [color=black] (-0.5,0.87)-- (0,0);
\draw [line width=2pt] (0,0)-- (1,0);
\draw [line width=2pt] (1,0)-- (1.5,0.87);
\draw [line width=2pt] (1.5,0.87)-- (1,1.73);
\draw [line width=2pt] (1,1.73)-- (0,1.73);
\draw [line width=2pt] (0,1.73)-- (-0.5,0.87);
\draw [line width=2pt] (-0.5,0.87)-- (0,0);
\draw [line width=2pt] (0,0)-- (1,0);
\draw [line width=2pt,color=red] (-0.5,0.87)-- (1,1.73);
\draw [line width=2pt,color=red] (-0.5,0.87)-- (1.5,0.87);
\draw [line width=2pt,color=red] (-0.5,0.87)-- (1,0);
\begin{scriptsize}
\fill [color=black] (0,0) circle (2pt);
\fill [color=black] (1,0) circle (2pt);
\fill [color=black] (1.5,0.87) circle (2pt);
\fill [color=black] (1,1.73) circle (2pt);
\fill [color=black] (0,1.73) circle (2pt);
\fill [color=black] (-0.5,0.87) circle (2pt);
\draw[color=black] (-0.2,-0.1) node {$1$};
\draw[color=black] (1.2,-0.1) node {$2$};
\draw[color=black] (1.7,0.87) node {$3$};
\draw[color=black] (1.2,1.74) node {$4$};
\draw[color=black] (-0.2,1.74) node {$5$};
\draw[color=black] (-0.7,0.87) node {$6$};
\end{scriptsize}
\end{tikzpicture}\begin{tikzpicture}[line cap=round,line join=round,x=1.0cm,y=1.0cm]
\clip(-1.22,-0.2) rectangle (2.23,1.89);
\draw [color=black] (0,0)-- (1,0);
\draw [color=black] (1,0)-- (1.5,0.87);
\draw [color=black] (1.5,0.87)-- (1,1.73);
\draw [color=black] (1,1.73)-- (0,1.73);
\draw [color=black] (0,1.73)-- (-0.5,0.87);
\draw [color=black] (-0.5,0.87)-- (0,0);
\draw [line width=2pt] (0,0)-- (1,0);
\draw [line width=2pt] (1,0)-- (1.5,0.87);
\draw [line width=2pt] (1.5,0.87)-- (1,1.73);
\draw [line width=2pt] (1,1.73)-- (0,1.73);
\draw [line width=2pt] (0,1.73)-- (-0.5,0.87);
\draw [line width=2pt] (-0.5,0.87)-- (0,0);
\draw [line width=2pt] (0,0)-- (1,0);
\draw [line width=2pt,color=red] (-0.5,0.87)-- (1,1.73);
\draw [line width=2pt,color=red] (1,0)-- (1,1.73);
\draw [line width=2pt,color=red] (-0.5,0.87)-- (1,0);
\begin{scriptsize}
\fill [color=black] (0,0) circle (2pt);
\fill [color=black] (1,0) circle (2pt);
\fill [color=black] (1.5,0.87) circle (2pt);
\fill [color=black] (1,1.73) circle (2pt);
\fill [color=black] (0,1.73) circle (2pt);
\fill [color=black] (-0.5,0.87) circle (2pt);
\draw[color=black] (-0.2,-0.1) node {$1$};
\draw[color=black] (1.2,-0.1) node {$2$};
\draw[color=black] (1.7,0.87) node {$3$};
\draw[color=black] (1.2,1.74) node {$4$};
\draw[color=black] (-0.2,1.74) node {$5$};
\draw[color=black] (-0.7,0.87) node {$6$};
\end{scriptsize}
\end{tikzpicture}
    \caption{Up to dihedral symmetry, every triangulation of the hexagon is one of these three: the snake (left), the shell (middle) and the triangle (right).}
    \label{fig:hexagontriangulations}
\end{figure}
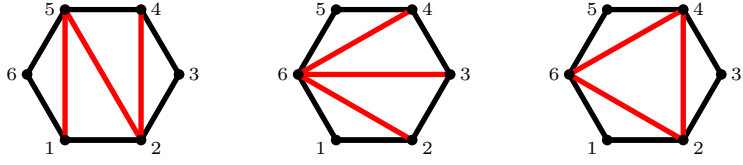
\begin{lem}\label{lem:snake3d}
    Up to dihedral symmetry, the secondary variables of every AV-rep\-resentation of $\mathrm{Adj}_3$ correspond to the edges of the snake triangulation in \Cref{fig:hexagontriangulations}.
\end{lem}

\begin{proof}
    Assume that this is not the case. Applying a dihedral symmetry we can then assume that the secondary variables are either $X_{26},X_{36},X_{46}$ or $X_{24},X_{26},X_{46}$. In both cases, the variables $X_{13},X_{15},X_{35}$ are primary. We calculate
    \begin{equation*}
        \frac{\partial A_3}{\partial X_{13}}\cdot \frac{\partial A_3}{\partial X_{15}}- A_3\cdot \frac{\partial^2 A_3}{\partial X_{13}\partial X_{15}}=-X_{14} X_{24} X_{25} X_{26} X_{36} X_{46}\cdot G
    \end{equation*}
    where $G$ is given by
    \begin{align*}
        G&=X_{14} X_{24} X_{25} X_{35} - X_{14} X_{24} X_{26} X_{35} + X_{24} X_{25} X_{35}^2 - X_{14} X_{26} X_{35}^2 -  X_{25} X_{26} X_{35}^2\\& + X_{14} X_{25} X_{35} X_{36} - X_{24} X_{26} X_{35} X_{36} + X_{25} X_{26} X_{35} X_{36}\\& +  X_{25} X_{35}^2 X_{36} - X_{26} X_{35}^2 X_{36} + X_{14} X_{24} X_{25} X_{46} - X_{14} X_{24} X_{26} X_{46}\\& +  X_{24} X_{25} X_{35} X_{46} - X_{14} X_{26} X_{35} X_{46} - X_{25} X_{26} X_{35} X_{46}\\& +  X_{14} X_{24} X_{36} X_{46} + X_{14} X_{35} X_{36} X_{46} + X_{24} X_{35} X_{36} X_{46} + X_{35}^2 X_{36} X_{46}.
    \end{align*}
    By \Cref{lem:abeer} it suffices to prove that $G$ is cannot be written as a product 
    \begin{equation*}
        G=(A+B\cdot X_{35})\cdot(C+D\cdot X_{35})
    \end{equation*}
    where the polynomials $A,B,C,D$ do not depend on $X_{35}$. 
    Assume for the sake of a contradiction that there exists such a factorization.
    Then $B\cdot D$ equals the coefficient of $X_{35}^2$, which is
    \begin{equation*}
        G_2=X_{24} X_{25} - X_{14} X_{26} - X_{25} X_{26} + X_{25} X_{36} - X_{26} X_{36} + X_{36} X_{46}.
    \end{equation*}
    We have that
    \begin{equation*}
        G_2|_{X_{14}=X_{24}=X_{46}=0}= - X_{25} X_{26} + X_{25} X_{36} - X_{26} X_{36}
    \end{equation*}
    which is clearly irreducible. Hence, without loss of generality, we can assume that $B=G_2$ and $D=1$. Because $G$ is homogeneous, each factor must be homogeneous as well, which shows that $C$ is a linear form. Because $G$ has degree one in all variables except for $X_{35}$, no variable that appears in $G_2$ can appear in $C$. Since every variable that appears in $G$, except for $X_{35}$, also appears in $G_2$, this implies that $C=0$. This is a contradiction because $G$ is not divisible by $X_{35}$.
\end{proof}

\begin{lem}\label{lem:reduceto4}
    Assume that there is an AV-representation of $\mathrm{Adj}_{n-2}$. Then, after applying a cyclic permutation of the variables, every diagonal of the $(n+1)$-gon that contains the vertex labeled by $n+1$ corresponds to a primary variable and the derivative with respect to all such primary variables is equal to $\mathrm{Adj}_{n-3}$. In particular, there exists an AV-representation for $\mathrm{Adj}_{n-3}$.
\end{lem}
\begin{proof}
    The secondary variables correspond to the diagonals in a certain triangulation $T_0
    $ of the $(n+1)$-gon. For every triangulation of the $(n+1)$-gon there exists a vertex which is not part of any diagonal in the triangulation. Hence, after a cyclic permutation of the vertices, we may assume that no diagonal in $T_0$ contains the vertex labeled by $n+1$. Then every diagonal of the $(n+1)$-gon that contains the vertex labeled by $n+1$ corresponds to a primary variable and the derivative in question equals to
    \begin{equation*}
        \sum_{T}\prod_{\substack{(i,j)\not\in T\\ i,j\in[n]}} X_{ij}.
    \end{equation*}
    where the sum is over all triangulations that do not contain any diagonal that contains the vertex $n+1$. These are in natural bijection with the triangulations of the $n$-gon which proves that the derivative in question is equal to $\mathrm{Adj}_{n-3}$. The last statement now follows from \Cref{rem:derivativeAVrep}.
\end{proof}
\begin{figure}
    \begin{tikzpicture}[line cap=round,line join=round,x=1.0cm,y=1.0cm]
\clip(-1.99,-0.3) rectangle (2.9,2.43);
\draw [line width=2pt] (-0.4,1.76)-- (0.5,2.19);
\draw [line width=2pt] (0.5,2.19)-- (1.4,1.76);
\draw [line width=2pt] (1.4,1.76)-- (1.62,0.78);
\draw [line width=2pt] (1.62,0.78)-- (1,0);
\draw [line width=2pt] (1,0)-- (0,0);
\draw [line width=2pt] (0,0)-- (-0.62,0.78);
\draw [line width=2pt] (-0.62,0.78)-- (-0.4,1.76);
\draw [line width=2pt,color=red] (0,0)-- (-0.4,1.76);
\draw [line width=2pt,color=red] (-0.4,1.76)-- (1,0);
\draw [line width=2pt,color=red] (-0.4,1.76)--(1.62,0.78);
\draw [line width=2pt,color=red] (1.62,0.78)-- (0.5,2.19);
\begin{scriptsize}
\fill [color=black] (0,0) circle (2pt); 
\fill [color=black] (1,0) circle (2pt); 
\fill [color=black] (1.62,0.78) circle (2pt); 
\fill [color=black] (1.4,1.76) circle (2pt); 
\fill [color=black] (0.5,2.19) circle (2pt); 
\fill [color=black] (-0.4,1.76) circle (2pt); 
\fill [color=black] (-0.62,0.78) circle (2pt); 
\draw[color=black] (-0.2,-0.1) node {$1$};
\draw[color=black] (1.2,-0.1) node {$2$};
\draw[color=black] (1.82,0.78) node {$3$};
\draw[color=black] (1.6,1.76) node {$4$};
\draw[color=black] (0.5,2.35) node {$5$};
\draw[color=black] (-0.6,1.76) node {$6$};
\draw[color=black] (-0.82,0.78) node {$7$};
\end{scriptsize}
\end{tikzpicture}\begin{tikzpicture}[line cap=round,line join=round,x=1cm,y=1cm]
\clip(-1.99,-0.3) rectangle (2.9,2.43);
\draw [line width=2pt] (-0.4,1.76)-- (0.5,2.19);
\draw [line width=2pt] (0.5,2.19)-- (1.4,1.76);
\draw [line width=2pt] (1.4,1.76)-- (1.62,0.78);
\draw [line width=2pt] (1.62,0.78)-- (1,0);
\draw [line width=2pt] (1,0)-- (0,0);
\draw [line width=2pt] (0,0)-- (-0.62,0.78);
\draw [line width=2pt] (-0.62,0.78)-- (-0.4,1.76);
\draw [line width=2pt,color=red] (0,0)-- (-0.4,1.76);
\draw [line width=2pt,color=red] (0,0)-- (1.62,0.78);
\draw [line width=2pt,color=red] (0,0)--(1.4,1.76);
\draw [line width=2pt,color=red] (1.4,1.76)-- (-0.4,1.76);
\begin{scriptsize}
\fill [color=black] (0,0) circle (2pt); 
\fill [color=black] (1,0) circle (2pt); 
\fill [color=black] (1.62,0.78) circle (2pt); 
\fill [color=black] (1.4,1.76) circle (2pt); 
\fill [color=black] (0.5,2.19) circle (2pt); 
\fill [color=black] (-0.4,1.76) circle (2pt); 
\fill [color=black] (-0.62,0.78) circle (2pt); 
\draw[color=black] (-0.2,-0.1) node {$1$};
\draw[color=black] (1.2,-0.1) node {$2$};
\draw[color=black] (1.82,0.78) node {$3$};
\draw[color=black] (1.6,1.76) node {$4$};
\draw[color=black] (0.5,2.35) node {$5$};
\draw[color=black] (-0.6,1.76) node {$6$};
\draw[color=black] (-0.82,0.78) node {$7$};
\end{scriptsize}
\end{tikzpicture}\begin{tikzpicture}[line cap=round,line join=round,x=1.0cm,y=1.0cm]
\clip(-1.99,-0.3) rectangle (2.9,2.43);
\draw [line width=2pt] (-0.4,1.76)-- (0.5,2.19);
\draw [line width=2pt] (0.5,2.19)-- (1.4,1.76);
\draw [line width=2pt] (1.4,1.76)-- (1.62,0.78);
\draw [line width=2pt] (1.62,0.78)-- (1,0);
\draw [line width=2pt] (1,0)-- (0,0);
\draw [line width=2pt] (0,0)-- (-0.62,0.78);
\draw [line width=2pt] (-0.62,0.78)-- (-0.4,1.76);
\draw [line width=2pt,color=red] (0,0)-- (-0.4,1.76);
\draw [line width=2pt,color=red] (0,0)-- (0.5,2.19);
\draw [line width=2pt,color=red] (1,0)-- (0.5,2.19);
\draw [line width=2pt,color=red] (1,0)-- (1.4,1.76);
\begin{scriptsize}
\fill [color=black] (0,0) circle (2pt); 
\fill [color=black] (1,0) circle (2pt); 
\fill [color=black] (1.62,0.78) circle (2pt); 
\fill [color=black] (1.4,1.76) circle (2pt); 
\fill [color=black] (0.5,2.19) circle (2pt); 
\fill [color=black] (-0.4,1.76) circle (2pt); 
\fill [color=black] (-0.62,0.78) circle (2pt); 
\draw[color=black] (-0.2,-0.1) node {$1$};
\draw[color=black] (1.2,-0.1) node {$2$};
\draw[color=black] (1.82,0.78) node {$3$};
\draw[color=black] (1.6,1.76) node {$4$};
\draw[color=black] (0.5,2.35) node {$5$};
\draw[color=black] (-0.6,1.76) node {$6$};
\draw[color=black] (-0.82,0.78) node {$7$};
\end{scriptsize}
\end{tikzpicture}
    \caption{Three triangulations of the heptagon.}
    \label{fig:heptagonsnake}
\end{figure}
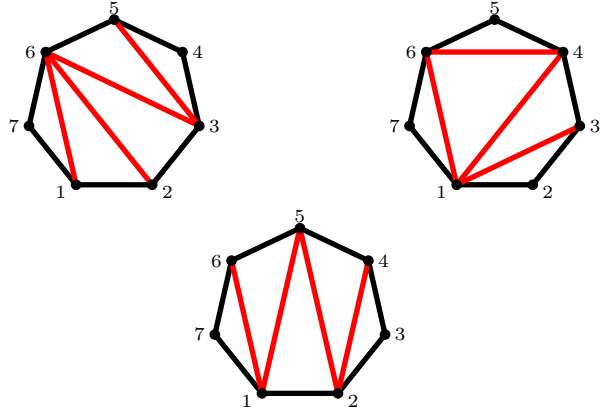
\begin{proof}[Proof of \Cref{thm:noavrep}]
    By \Cref{lem:reduceto4} it suffices to prove the claim for $n-3=4$. Assume for the sake of a contradiction that $\mathrm{Adj}_4$ has an AV-representation. We first determine the secondary variables of the AV-representation. By \Cref{lem:reduceto4} we can assume that they include the variable $X_{16}$. Again by \Cref{lem:reduceto4} it follows from \Cref{lem:snake3d} that the remaining secondary variables correspond to a triangulation of the hexagon with vertices labeled by $1,\ldots,6$ that is, up to dihedral symmetry, the snake from \Cref{fig:hexagontriangulations}. Up to dihedral symmetry, we are then left with the three triangulations from \Cref{fig:heptagonsnake}. Applying the same argument for the vertices labeled by $4$ and $5$ instead of $7$ excludes the upper two triangulations from \Cref{fig:heptagonsnake} because the induced triangulations of the hexagons obtained by removing these vertices do not agree with the snake triangulation.
    Hence we are left with the lower triangulation from \Cref{fig:heptagonsnake}. In this case, the primary variables include $X_{13},X_{27},X_{35},X_{37},X_{47},X_{57}$.
    We consider the derivative $F$ of $\mathrm{Adj}_4$ with respect to the primary variables $X_{27}, X_{37}, X_{47}$. By \Cref{rem:derivativeAVrep} the existence of an AV-representation of $\mathrm{Adj}_4$ implies the existence of an AV-representation of $F$. The monomials of $F$ correspond to triangulations of the heptagon that do not contain the diagonals $27, 37$ and $47$. If such a triangulation does not contain $57$, then it consists of the diagonal $16$ and a triangulation of the hexagon with vertices labeled by $1,\ldots,6$. If such a triangulation contains the diagonal $57$, then it also contains the diagonal $15$ and a triangulation of the pentagon with vertices labeled by $1,\ldots,5$. This shows that
    \begin{equation*}
        F= X_{57} \cdot \mathrm{Adj}_3 + X_{16} X_{26} X_{36} X_{46} \cdot \mathrm{Adj}_2.
    \end{equation*}
    Now we calculate the Rayleigh difference of $F$ with respect to $X_{13}$ and $X_{57}$ as
    \begin{equation*}
        \frac{\partial F}{\partial X_{13}}\cdot \frac{\partial F}{\partial X_{57}}- F\cdot \frac{\partial^2 F}{\partial X_{13}\partial X_{57}}=-X_{14} X_{15} X_{16} X_{24} X_{25} X_{26} X_{36} X_{46}\cdot G
    \end{equation*}
    where $G$ is the polynomial from the proof of \Cref{lem:snake3d}. There we have seen that $G$ cannot we written as the product of two polynomials that are affine in the primary variable $X_{35}$. Thus by \Cref{lem:abeer} the polynomial $F$ and hence $\mathrm{Adj}_4$ does not have an AV-representation.
\end{proof}

Note that \Cref{thm:noavrep} does not imply that $\mathrm{Adj}_{n-3}$, $n\geq7$, does not have a determinantal representation in the sense of \Cref{def:detrep}, cf. \Cref{rem:comparison}. However, in view of the negative results from \Cref{sec:smooth} we do not expect such a representation to exist. Further note that if there existed such a representation
\begin{equation*}
    \mathrm{Adj}_{n-3}=\det\left(\sum X_{ij}\cdot M_{ij}\right)
\end{equation*}
for suitable scalar matrices $M_{ij}$, then it would not reflect the multi-affine structure of $\mathrm{Adj}_{n-3}$, $n\geq7$, in the sense that the $M_{ij}$ have rank one: otherwise, one could transform it to an AV-representation by multiplying suitable invertible matrices from left and right. We do not find it likely that such a representation could serve as an appropriate certificate for the ``hidden zeros'' of $\mathrm{Adj}_{n-3}$ as favored in \cite{arkani2024hidden}.
\begin{ex}
    Already $\mathrm{Adj}_{2}$ does not have a determinantal representation 
    \begin{equation*}
    \mathrm{Adj}_{2}=\det\left(\sum X_{ij}\cdot M_{ij}\right)
    \end{equation*}
     that reflects the multi-affine structure in the sense that the $M_{ij}$ have rank one. Indeed, assume for the sake of a contradiction that there is such a representation. We can transform it, by multiplying suitable invertible matrices from left and right, to an AV-representation whose primary variables include $X_{13}$ and $X_{14}$. We have
     \begin{equation}\label{eq:adj2wronskian}
         \frac{\partial \mathrm{Adj}_{2}}{\partial X_{13}}\cdot \frac{\partial \mathrm{Adj}_{2}}{\partial X_{14}}- \mathrm{Adj}_{2}\cdot \frac{\partial^2 \mathrm{Adj}_{2}}{\partial X_{13}\partial X_{14}}=-X_{24} X_{25} X_{35} (X_{24}-X_{25}+X_{35}).
     \end{equation}
     According to \Cref{lem:abeer}, this polynomial factors into two polynomials that are $2\times2$ subdeterminants of our hypothetical determinantal representation. But if all $M_{ij}$ have rank one, such subdeterminants will  be multi-affine of degree two. Obviously, the polynomial in \Cref{eq:adj2wronskian} does not admit such a factorization.
\end{ex}

\bibliographystyle{alpha}
\bibliography{bibliography}

\end{document}